\documentclass[11pt]{article}

\usepackage[T1]{fontenc}
\usepackage[utf8]{inputenc}
\usepackage{amsmath,amssymb,amsfonts,amsthm,mathtools}
\usepackage{mathrsfs}
\usepackage{geometry}
\usepackage{hyperref}
\usepackage{enumitem}
\usepackage{microtype}
\usepackage{xcolor}
\usepackage[normalem]{ulem}
\usepackage{graphicx}
\usepackage{float}
\usepackage{cite}
\usepackage{subfigure}
\usepackage{color}
\usepackage{mdwlist}
\usepackage{authblk}

\hypersetup{colorlinks=true,linkcolor=blue,citecolor=blue,urlcolor=blue}
\numberwithin{equation}{section}

\newtheorem{theorem}{Theorem}[section]
\newtheorem{lemma}[theorem]{Lemma}
\newtheorem{proposition}[theorem]{Proposition}
\newtheorem{corollary}[theorem]{Corollary}
\newtheorem{definition}[theorem]{Definition}
\newtheorem{assumption}[theorem]{Assumption}
\theoremstyle{remark}
\newtheorem{remark}[theorem]{Remark}

\newcommand{\R}{\mathbb R}

\newcommand{\dd}{\,d}

\newcommand{\norm}[1]{\left\|#1\right\|}
\newcommand{\ip}[2]{\left\langle #1,#2\right\rangle}
\newcommand{\pd}{p^{\dagger}}

\newcommand{\Sset}{\mathcal S}
\newcommand{\Lpwt}{L^p_w(0,R)}

\newcommand{\Ltwo}{L^2_w(0,R)}
\newcommand{\phip}{\varphi_p}
\newcommand{\phid}{\varphi_{\pd}}

\newcommand{\add}[1]{{\color{blue}#1}}

\title{Finite inverse nodal problems for singular weighted Sturm--Liouville
 equations: variational selection and\\ spectral matching\footnote
{Zhibo Cheng was supported by the National Natural Science
Foundation of China (12671192) and Natural Science Foundation of
Henan (262300421839). Yonghui Xia was supported by the National
Natural Science Foundation of China (12571176) and Natural Science
Foundation of Guangdong Province (2026A1515011073).}}
\author[1]{Zhibo Cheng}
\affil[1]{\small \itshape School of Mathematica and Information
Science, Henan Polytechnic University, Jiaozuo, China}

\author[2]{Yuchao He} \affil[2]{\small \itshape School of
Mathematical Science, Zhejiang Normal University, Jinhua, China}

\author[1]{Sunan Wang}

\author[3\footnote{Corresponding author. E-mail:~yhxia@zjnu.cn; xiadoc@163.com }]{Yonghui Xia}
\affil[3]{\small \itshape School of Mathematics, Foshan University,
Foshan, China}

\date{}

\begin{document}
\maketitle

\begin{abstract}
The recent work \cite{HeWuXiaZhang2025} introduced a finite-data
inverse nodal framework for regular one-dimensional Sturm-Liouville
operators. Nevertheless, formulating a multidimensional inverse
nodal theory for Schr\"odinger operators remains a challenging open
problem. This paper addresses this gap by establishing a finite
radial inverse nodal theory for the radial spectral branch of
Schr\"odinger operators in balls.
 In contrast with the interval case, the nodal data are the radii of
spherical nodal hypersurfaces of radial eigenfunctions.  The exact
radial reduction preserves this PDE nodal geometry and leads
naturally to a singular weighted Sturm--Liouville problem with
weight $w(r)=r^{n-1}$, rather than to a regular interval model. We
minimize the PDE distance $\|Q-Q_0\|_{L^p(B_R)}$  , equivalently
    the weighted distance $\|q-q_0\|_{L^p_w(0,R)}$, among all radial potentials matching finitely
many prescribed spherical nodal radii. For the associated nodal
constraint sets, we prove spectral regularity, nodal
differentiability, weak closedness, exact realization, and a
finite-codimensional
 $C^1$
    manifold structure. Crucially, without the variational selection imposed by
\(q_0\), the prescribed finite nodal radii fail to determine the potential uniquely:
 the admissible potentials are invariant under constant shifts, and the finite nodal constraints
  yield locally infinite-dimensional level sets at regular points. The variational selection,
  however, restores rigidity.
  Every nontrivial minimizer satisfies piecewise
nonlinear radial Schr\"{o}dinger equations, a mass-balance condition across a
prescribed nodal sphere, and a radial energy identity with dissipation.  The
Liouville normal form explains the relation with one-dimensional inverse nodal
problems, but the inverse-square singularity at the origin and the PDE-natural
weighted metric prevent a direct reduction to the regular interval theory.  In
the constant benchmark case we obtain a zero-distance alternative, a general
$L^p$ near-benchmark first-order selection law, perturbative uniqueness in the
Hilbert radial regime $p=2$, $n=2,3$, and principal-eigenvalue spectral matching
reductions for the remaining global rigidity problem.
\end{abstract}

\noindent\textbf{Keywords.} Inverse nodal problem; radial Schr\"{o}dinger operator;
weighted Sturm--Liouville problem; finite nodal data;
{variational
selection; spectral matching}.\\
\textbf{MSC 2020.} 34B24; 34L05; 34L20; 35J10; 35P15; 49K20.

\section{Introduction and PDE formulation}\label{sec:intro-main}

Inverse spectral theory asks to what extent an operator, a
potential, or a domain can be recovered from spectral information.
The geometric version of this question is epitomized by Kac's
question on hearing the shape of a drum \cite{Kac1966}; related
spectral-geometric phenomena and sharp eigenvalue estimates have led
to a rich literature on what spectral data can and cannot determine
\cite{AshbaughBenguria1989,AshbaughBenguria1992,BerardWebb2022}.
Besides eigenvalues, another important source of inverse information
is provided by the nodal sets of eigenfunctions.  For
one-dimensional Sturm--Liouville operators the zeros of an
eigenfunction form an ordered finite set of nodal points, and
inverse nodal theory studies the recovery of coefficients from such
nodal data.  Classical uniqueness results usually require an
infinite, or asymptotically dense, family of nodal points; see, for
instance,
\cite{Borg1946,Levitan1987,McLaughlin1988,PoschelTrubowitz1986,Yang1997,
Zettl2005}.  The inverse nodal viewpoint has since been extended in
several directions, including discontinuous or nonstandard
Sturm--Liouville settings and reconstruction from zeros of selected
eigenfunctions
\cite{BennewitzBrownWeikard2020,BrowneSleemanWeikard1996,ChenChengLaw2011,Hald1984,HaldMcLaughlin1996,
LawTsay2001,WangYurko2016,Yang2013,Yang2001}.  A recent finite-data
formulation takes a variational point of view: for a prescribed
finite nodal datum and a target potential $q_0$, one minimizes the
$L^p$-distance to $q_0$ among all potentials realizing that datum.
In the regular interval case this approach leads, through a nodal
derivative formula and the Lagrange multiplier rule, to a nonlinear
Schr\"{o}dinger-type critical equation \cite{HeWuXiaZhang2025};
closely related optimization problems for Sturm--Liouville spectra
and nodes have also been studied in \cite{
ChuMengZhang2023,ChuMengWangZhang2024,ChuMengZhang2026,IlyasovValeev2019,ValeevIlyasov2018}.

Passing from ordinary differential equations to partial differential
equations changes the nature of nodal data in a fundamental way.  In
one space dimension, nodal points are isolated and naturally ordered,
as in the classical oscillation theory for second-order equations
\cite{Arnold1992,BirkhoffRota1969}.  In a general multidimensional
domain, however, the nodal set of an eigenfunction is typically a
hypersurface with nontrivial geometry, and there is no immediate
analogue of the ordered nodal sequence; this is one of the central
features of the nodal-domain theory initiated by Courant and Hilbert
and developed in many later works
\cite{Arnold2011,BerardHelffer2020,CourantHilbert1953}.  The present
paper focuses on a situation in which this geometric difficulty can
be handled without losing its PDE origin.
We consider Schr\"{o}dinger operators in a ball and restrict the
potential to be radial.  Then the radial nodal sets are concentric
spheres, and their radii form an ordered finite sequence.  Thus the
multidimensional nodal geometry is converted, by an exact radial
reduction, into a singular weighted Sturm--Liouville problem.

More precisely, let $B_R\subset \mathbb R^n$, $n\ge 2$, be the ball
of radius $R>0$, and consider the Dirichlet Schr\"{o}dinger problem
\[
        -\Delta u+Q(x)u=\lambda u\quad \hbox{in }B_R,
        \qquad u=0\quad \hbox{on }\partial B_R .
\]
We assume that the potential is radial, namely
\[
        Q(x)=q(|x|).
\]
The radial subspace of $L^2(B_R)$ is then invariant under the
operator, and the corresponding radial eigenvalue problem defines a
distinguished spectral branch.  If $u(x)=y(r)$ with $r=|x|$, then
\[
        \Delta u(x)=y''(r)+\frac{n-1}{r}y'(r),\qquad 0<r<R.
\]
Consequently the radial eigenvalue equation is
\begin{equation}\label{eq:radial-ode-intro}
-y''(r)-\frac{n-1}{r}y'(r)+q(r)y(r)=\lambda y(r),
        \qquad 0<r<R,
\end{equation}
with the regularity condition $y'(0)=0$ and the Dirichlet condition
$y(R)=0$.  Equivalently, with
\[
        w(r)=r^{n-1},
\]
we obtain the singular weighted Sturm--Liouville problem
\begin{equation}\label{eq:radial-sl-main}
        -\big(wy'\big)' + wq y=\lambda wy,
        \qquad 0<r<R .
\end{equation}
The natural Hilbert space is $L^2_w(0,R)=L^2((0,R),w(r)\,dr)$.  Up
to the constant factor $|\mathbb S^{n-1}|$, this is precisely the
radial part of $L^2(B_R)$, since
\[
        \int_{B_R}|u(x)|^2\,dx
        = |\mathbb S^{n-1}|\int_0^R |y(r)|^2 r^{n-1}\,dr .
\]
The same correspondence holds for the Dirichlet form.  Hence the
weighted problem \eqref{eq:radial-sl-main} is not an auxiliary model
separated from the PDE; it is the exact radial spectral problem
associated with the Schr\"{o}dinger operator in the ball.

This reduction also gives a precise meaning to finite radial inverse
nodal data.  Let $E_m(\cdot;q)$ be the $m$-th radial eigenfunction
of \eqref{eq:radial-sl-main}.  Its interior zeros are denoted by
\[
        0<T_{1,m}(q)<\cdots<T_{m-1,m}(q)<R .
\]
{The corresponding radial eigenfunction of the PDE is represented, up to a
nonzero normalization constant, by the radial representative
\(E_m(|x|;q)\).  Equivalently, under the radial unitary identification
\[
        \mathcal U:L^2_w(0,R)\to L^2(B_R),
        \qquad
        (\mathcal U y)(x)=|\mathbb S^{n-1}|^{-1/2}y(|x|),
\]
the function \(\mathcal U E_m(\cdot;q)\) is the \(L^2(B_R)\)-normalized
radial eigenfunction.  Since multiplication by a nonzero constant does not
change the zero set, the corresponding radial PDE eigenfunction satisfies}
\[
        u_m(x;q)=0
        \quad\Longleftrightarrow\quad
        |x|=T_{j,m}(q)\quad \hbox{for some }j,
\]
and the nodal set of $u_m$ in the ball is the union of the
concentric spheres
\[
        \bigcup_{j=1}^{m-1}\{x\in B_R:\ |x|=T_{j,m}(q)\} .
\]
Recovering the radial potential from the radii of these nodal
spheres is therefore exactly the same as recovering $q$ from the
zeros of the weighted radial Sturm--Liouville eigenfunction.  This
is the finite radial inverse nodal problem studied in the sequel.

There are essential differences from the regular one-dimensional
interval problem.  The endpoint $r=0$ is singular in the ordinary
Sturm--Liouville sense, the natural measure is
$w(r)\,dr=r^{n-1}\,dr$, and the boundary condition at the origin is
a radial regularity condition rather than a standard separated
boundary condition.  This places the problem closer in spirit to
singular and weighted spectral problems than to a regular interval
model \cite{EckhardtKostenko2016,Teschl2014}.  Moreover, the radial term $(n-1)r^{-1}y'$ has a
genuine effect on the variational critical equation.  When the
Euler--Lagrange equation associated with the inverse nodal
optimization is written on a nodal interval, the corresponding
radial energy satisfies a dissipative identity rather than a
conserved Hamiltonian first integral.  Thus the phase-plane
reconstruction available in the regular interval case cannot be
transferred directly to the radial setting.  The analysis must
instead be carried out in weighted spaces and with estimates adapted
to the singular endpoint.

A useful comparison with the regular interval problem is provided by the
Liouville transformation.   Let
\[
        \alpha=\frac{n-1}{2},\qquad v(r)=r^\alpha y(r).
\]
Then \eqref{eq:radial-ode-intro} is formally transformed into
\begin{equation}\label{eq:liouville-intro}
        -v''(r)+\left(q(r)+\frac{(n-1)(n-3)}{4r^2}\right)v(r)
        =\lambda v(r),\qquad 0<r<R .
\end{equation}
{Consequently, the Liouville normal form should be regarded as a comparison
device rather than as the basic variational setting of the problem.  Although
the transformation removes the first-order radial term and preserves the nodal
radii, the resulting operator is still singular at the origin.  Indeed, its
effective potential contains the fixed inverse-square term
\((n-1)(n-3)/(4r^2)\), and the admissible behavior at \(r=0\) is determined by
the finite-energy radial condition inherited from the ball.  Moreover, the
natural distance in the optimization problem remains
\(\|Q-Q_0\|_{L^p(B_R)}\), equivalently
\(\|q-q_0\|_{L^p_w(0,R)}\), rather than the unweighted \(L^p(0,R)\) distance.}
Consequently, the weighted radial formulation is not merely a change of
variables; it is the natural formulation of the finite inverse nodal problem
for spherical nodal sets of the radial Schr\"{o}dinger operator.

{The weighted radial equation
\eqref{eq:radial-sl-main} is the exact radial representative of the
Schr\"{o}dinger operator in the ball and carries the natural radial measure
\(w(r)\,dr=r^{n-1}\,dr\).  The Liouville normal form
\eqref{eq:liouville-intro} is retained only for comparison with regular
one-dimensional problems.  Since these two formulations have already been
explained here, the separate section devoted to them is removed in order to
avoid repetition and to let the paper proceed directly to the spectral and
nodal regularity theory.}

We now formulate the finite radial inverse nodal optimization problem first in
PDE language.  Let
\[
        \mathcal R L^p(B_R)
        =\{Q\in L^p(B_R): Q(x)=q(|x|)\ \hbox{for some }q\}.
\]
For a radial potential $Q(x)=q(|x|)$, let
\[
        \mathcal T_m(Q)=\big(T_{1,m}(q),\ldots,T_{m-1,m}(q)\big)
\]
be the vector of radii of the spherical nodal sets of the $m$-th radial
PDE eigenfunction.  Given a radial target $Q_0(x)=q_0(|x|)$ and an
admissible vector
\[
        \tau=(\tau_1,\ldots,\tau_{m-1}),
        \qquad 0<\tau_1<\cdots<\tau_{m-1}<R,
\]
the PDE form of the finite radial inverse nodal optimization problem is
\begin{equation}\label{eq:pde-full-problem}
        \min\{\|Q-Q_0\|_{L^p(B_R)}:
        Q\in \mathcal R L^p(B_R),\ \mathcal T_m(Q)=\tau\}.
\end{equation}
For a single prescribed spherical nodal radius $T^*\in(0,R)$, the corresponding
one-node problem is defined by imposing $T_{i,m}(q)=T^*$ in place of
$\mathcal T_m(Q)=\tau$.

The weighted formulation used in the sequel is exactly equivalent to
\eqref{eq:pde-full-problem}.  Indeed, polar integration gives
\begin{equation}\label{eq:pde-weighted-norm}
        \|Q-Q_0\|_{L^p(B_R)}^p
        =|\mathbb S^{n-1}|\int_0^R |q(r)-q_0(r)|^p r^{n-1}\,dr .
\end{equation}
Thus, up to the harmless constant $|\mathbb S^{n-1}|^{1/p}$, the natural PDE
metric on radial potentials is the weighted norm.  Let $1<p<\infty$ and set
\[
        L^p_w(0,R)=L^p((0,R),w(r)\,dr),
        \qquad
        \|q\|_{p,w}
        =\left(\int_0^R |q(r)|^p w(r)\,dr\right)^{1/p} .
\]
Throughout the main part of the paper we assume
\[
        p>\frac{n}{2} .
\]
This condition ensures that the potential term is controlled by the Dirichlet
form and that the radial eigenfunctions have the local regularity needed to
differentiate their nodal points.  We denote the conjugate exponent of $p$ by
\[
        p^\dagger=\frac{p}{p-1} .
\]

For $q\in L^p_w(0,R)$, let $\lambda_m(q)$ be the $m$-th radial eigenvalue and
let $E_m(\cdot;q)$ be the associated real $L^2_w$-normalized radial
eigenfunction.  The corresponding weighted nodal map is
\[
        T_m(q)=\big(T_{1,m}(q),\ldots,T_{m-1,m}(q)\big).
\]
We define
\[
        \Omega_m=\{\tau:\ 0<\tau_1<\cdots<\tau_{m-1}<R\},
\]
and
\[
        \mathcal S^{(m)}_{\tau}
        =\{q\in L^p_w(0,R):\ T_m(q)=\tau\} .
\]
For a single prescribed node $T^*\in(0,R)$, we write
\[
        \mathcal S_{i,m}(T^*)
        =\{q\in L^p_w(0,R):\ T_{i,m}(q)=T^*\} .
\]
The two weighted constrained optimization problems considered in this paper are
\begin{equation}\label{eq:main-full-problem}
        \min_{q\in \mathcal S^{(m)}_{\tau}} \|q-q_0\|_{p,w}
\end{equation}
and
\begin{equation}\label{eq:main-single-problem}
        \min_{q\in \mathcal S_{i,m}(T^*)} \|q-q_0\|_{p,w} .
\end{equation}
In view of \eqref{eq:pde-weighted-norm}, these weighted problems are precisely
the PDE optimization problems restricted to radial potentials.  The weighted
norm is therefore not an auxiliary choice; it is the PDE-natural distance for
radial potentials in the ball.

The first step is to justify that this formulation is well posed in
the singular weighted setting.  We prove that the radial eigenvalues
are simple, that the $m$-th radial eigenfunction has exactly $m-1$
simple interior zeros, and that each nodal map $q\mapsto T_{i,m}(q)$
is locally $C^1$ on $L^p_w(0,R)$.  We also establish weak sequential
continuity of the nodal maps on bounded sets, a feature closely
connected with weak-topology continuity phenomena for differential
systems \cite{Zhang2008}.  This compactness property is essential
for the existence of minimizers in the finite-data variational
problem. A radial derivative formula for the nodal maps is then
derived and used as the basic input for the Lagrange multiplier
analysis.

Our second contribution concerns the realization and geometry of
finite radial nodal constraints.  We prove constructively that every
admissible finite nodal vector $\tau\in\Omega_m$ can be realized by
a bounded piecewise constant radial potential.  Thus the constraint
sets in \eqref{eq:main-full-problem} and
\eqref{eq:main-single-problem} are never empty.  We further show
that $\mathcal S^{(m)}_{\tau}$ is a $C^1$ submanifold of
$L^p_w(0,R)$ of codimension $m-1$, while $\mathcal S_{i,m}(T^*)$ is
a $C^1$ hypersurface.  These results place the finite inverse nodal
problem into a genuine infinite-dimensional constrained variational
framework.

The third part of the paper is devoted to variational selection.
{Before the target potential \(q_0\) and the
minimization criterion are imposed, the prescribed nodal radii define only a
raw nodal constraint.  Such data do not determine a unique potential: adding a
constant to the potential leaves all radial eigenfunctions, and hence all
radial nodal radii, unchanged.  In addition, for finite nodal maps, the
corresponding level sets are locally infinite-dimensional at regular points.}
Nevertheless, the minimization problem has a rigid Euler--Lagrange structure.  If a minimizer is not equal to the
target $q_0$, then it satisfies a Lagrange multiplier equation.  On
each nodal interval, a suitable rescaling of the optimal
eigenfunction solves a nonlinear radial Schr\"{o}dinger equation.  For
a single prescribed node, the nonlinear term has opposite signs on
the two sides of the node, and a mass-balance condition holds across
the prescribed nodal sphere.  This equation is the radial
counterpart of the critical equation arising in the regular interval
problem, but it also carries the radial dissipation produced by the
term $(n-1)r^{-1}U'$.

Finally, we analyze the benchmark case in which $m=2$, $i=1$, and
$q_0\equiv c$ is constant.  Let
\[
        T_c=T_{1,2}(c).
\]
We prove a zero-distance alternative: the constant potential is the
unique global optimizer exactly at its own nodal position $T_c$; if
$T^*\ne T_c$, then the optimal distance is strictly positive,
although the raw finite inverse nodal constraint remains infinitely
non-unique. For arbitrary $p>n/2$, the benchmark problem has a sharp
near-benchmark first-order selection law: every global minimizer,
whether or not it is unique, has the same leading profile determined
by the dual nodal gradient.  In the Hilbert regime $p=2$, $n=2,3$,
and under a local $C^2$ regularity assumption on the benchmark nodal
map, this first-order selection law is upgraded to uniqueness of the
variationally selected optimizer for all nontrivial prescribed nodes
sufficiently close to $T_c$.  We also reduce the remaining global
uniqueness or multiplicity question to principal-eigenvalue spectral
matching functionals on the radial nodal subintervals.

The main conclusions are organized in four parts, in order to emphasize the PDE
formulation, the constraint geometry, the variational selection mechanism, and
the benchmark spectral matching theory.

\begin{theorem}\label{thm:PDE-formulation}
Let $n\ge2$, $R>0$, and assume $p>n/2$.  Let $Q(x)=q(|x|)$ be a real radial
potential with $q\in L^p_w(0,R)$.  Then the radial spectral branch of
$-\Delta+Q$ in $B_R$ is exactly represented by the singular weighted problem
\eqref{eq:radial-sl-main}.  Its radial eigenvalues are simple, the $m$-th radial
PDE eigenfunction has exactly $m-1$ spherical nodal hypersurfaces, and their
radii are precisely
\[
        0<T_{1,m}(q)<\cdots<T_{m-1,m}(q)<R .
\]
Moreover, the PDE optimization problem \eqref{eq:pde-full-problem} is equivalent,
up to the harmless factor $|\mathbb S^{n-1}|^{1/p}$, to the weighted variational
problem \eqref{eq:main-full-problem}.  Thus the finite weighted inverse nodal
problem is the PDE-natural radial inverse problem for spherical nodal data.
\end{theorem}

\begin{theorem}\label{thm:constraint-geometry-main}
Let $n\ge2$, $R>0$, and assume $p>n/2$.
\begin{enumerate}[label=\textup{(\roman*)}]
\item Each nodal map $q\mapsto T_{i,m}(q)$ is locally $C^1$ on $L^p_w(0,R)$ and
weakly sequentially continuous on bounded sets.
\item Every admissible finite radial nodal vector $\tau\in\Omega_m$ is realized
by a bounded piecewise constant radial potential.  Consequently, the constraint
sets $\mathcal S^{(m)}_\tau$ and $\mathcal S_{i,m}(T^*)$ are nonempty.
\item The full nodal constraint set $\mathcal S^{(m)}_\tau$ is a $C^1$ submanifold
of $L^p_w(0,R)$ of codimension $m-1$, and $\mathcal S_{i,m}(T^*)$ is a $C^1$
hypersurface.
\item {Before imposing the variational selection by a
target potential \(q_0\), prescribed finite radial nodal radii do not determine
a unique potential.  The nodal data are invariant under constant shifts of the
potential.  Moreover, at regular points of any finite nodal map, the
corresponding level set is locally infinite-dimensional.}
\end{enumerate}
\end{theorem}

\begin{theorem}\label{thm:variational-selection-main}
Let $q_0\in L^p_w(0,R)$.  The full-vector and one-node optimization problems
\eqref{eq:main-full-problem} and \eqref{eq:main-single-problem} admit minimizers.
If a minimizer is not equal to $q_0$, then it satisfies a Lagrange multiplier
equation.  On each nodal interval, a suitable rescaling of the optimal radial
eigenfunction solves a piecewise nonlinear radial Schr\"{o}dinger equation.  In the
single-node case the nonlinear term has opposite signs on the two sides of the
prescribed nodal sphere, and a cross-node mass-balance condition holds.  If
$q_0$ is locally absolutely continuous, the corresponding radial energy satisfies
a dissipative identity; in particular, for constant $q_0$ the Hamiltonian first
integral of the regular interval problem is replaced by a strictly dissipative
radial energy law.
\end{theorem}

\begin{theorem}\label{thm:benchmark-matching-main}
Assume $m=2$, $i=1$, and $q_0\equiv c$.  Let $T_c=T_{1,2}(c)$.
\begin{enumerate}[label=\textup{(\roman*)}]
\item The constant potential is the unique global optimizer exactly at its own
nodal position $T_c$.  If $T^*\ne T_c$, then the optimal distance is strictly
positive, while {the unoptimized nodal constraint still contains infinitely many
admissible potentials}.
\item For arbitrary $p>n/2$, the benchmark variational problem has a sharp
near-benchmark first-order selection law: every global minimizer, whether or not
it is unique, has the same leading profile determined by the dual nodal gradient.
\item In the Hilbert case $p=2$, $n=2,3$, and under a local $C^2$ regularity
assumption on the benchmark nodal map near $c$, the first-order selection law is
upgraded to uniqueness of the variationally selected optimizer for all nontrivial
prescribed nodes sufficiently close to $T_c$.
\item In the Hilbert benchmark setting, the remaining global
uniqueness/multiplicity question is reduced to the minimizer structure of
principal-eigenvalue spectral matching functionals on the radial nodal
subintervals.  In particular, for a finite nodal vector the matching functional
is obtained by matching the principal eigenvalues on all nodal subintervals; a
strict matching minimum yields uniqueness, whereas multiple matching minima yield
multiplicity.
\end{enumerate}
\end{theorem}

Several comments clarify the scope of these results.  First, the radial reduction
is exact for radial potentials and radial eigenfunctions: the nodal radii in
\eqref{eq:radial-sl-main} are precisely the radii of the spherical nodal sets of
the original Schr\"{o}dinger operator in the ball.  Second, the constructive
realization theorem shows that arbitrary admissible finite spherical nodal radii
are compatible with some bounded radial potential; no hidden compatibility
condition is imposed.  Third, the results separate two different issues.
{Before the target potential \(q_0\) and
the variational selection criterion are introduced, nodal radii alone do not
select a unique potential.  Constant shifts already leave all nodal radii
unchanged, and finite nodal constraints have locally infinite-dimensional level
sets at regular points.} The variational problem, however, selects potentials
with a rigid Euler--Lagrange structure and, near the constant benchmark in the
Hilbert regime, selects a unique optimizer.  This distinction is the central
mechanism behind the finite radial inverse nodal theory developed in this paper.

{We close the introduction by emphasizing the main
contributions of the paper and by clarifying their relation to the
finite inverse nodal problem on a regular interval.  First, the
nodal data considered here come from the PDE geometry of the
Schr\"{o}dinger operator in a ball.  Under the radial symmetry
assumption, the nodal sets of radial eigenfunctions are spherical
hypersurfaces, and the finite data are precisely their radii.  Thus
the prescribed nodes are not isolated interval zeros introduced a
priori, but geometrically meaningful nodal radii inherited from the
original PDE.

Second, the paper develops a singular weighted variational framework adapted to
this radial PDE setting.  The exact radial reduction leads to a weighted
Sturm--Liouville problem with $w(r)=r^{n-1}$ and with a singular endpoint at
$r=0$.  The weighted norm in $L^p_w(0,R)$ is exactly the radial expression of the
natural $L^p(B_R)$-distance between radial potentials.  Within this framework we
establish spectral and nodal regularity, weak closedness, exact realization of
arbitrary admissible finite nodal vectors, and the finite-codimensional $C^1$
geometry of the corresponding constraint sets.

Third, the paper separates the intrinsic non-uniqueness of finite nodal data from
the rigidity created by variational selection.  Finite radial nodal radii do not
determine the potential; indeed, constant shifts preserve all nodal radii, and
the finite-data level sets are locally infinite-dimensional at regular points.
Nevertheless, once a target potential is prescribed and the distance to it is
minimized, the selected potentials satisfy a rigid Euler--Lagrange structure.  In
particular, after a suitable rescaling of the optimal radial eigenfunction, one
obtains piecewise nonlinear radial Schr\"{o}dinger equations and, in the
single-node case, a cross-node mass-balance condition.

Finally, the radial geometry produces an obstruction to the Hamiltonian
reconstruction used in the regular interval problem and leads instead to a
spectral matching mechanism.  In the interval setting, the constant-target
critical equation admits a Hamiltonian first integral and can be analyzed by
phase-plane methods.  In the radial setting, the term $(n-1)r^{-1}U'$ produces a
dissipative radial energy identity.  This explains why the interval
reconstruction mechanism cannot be transferred directly to the PDE radial
problem.  In the constant benchmark regime we therefore develop a different
rigidity theory, including a near-benchmark first-order selection law,
perturbative uniqueness in the Hilbert radial setting, and a principal-eigenvalue
spectral matching reduction for the remaining global problem.

These points form the conceptual route of the paper.  The PDE nodal
geometry leads to a singular weighted radial formulation; finite
data are intrinsically non-unique; and the constrained variational
principle selects potentials with a rigid radial Euler--Lagrange and
spectral matching structure.  The subsequent sections develop these
ingredients in detail.}

\section{Spectral and nodal regularity in radial spaces}\label{sec:spectral-nodal}

We identify a radial function on $(0,R)$ with the corresponding radial function
on $B_R$.  The natural Hilbert space is $\Ltwo$ with inner product
\[
        \ip{u}{v}_w=\int_0^R u(r)v(r)w(r)\dd r.
\]
Set
\[
H^1_{w,R}(0,R)=\left\{y:\int_0^R\big(|y'|^2+|y|^2\big)w\dd r<\infty,
\quad y(R)=0\right\}.
\]
This is the radial representative of $H^1_0(B_R)$ restricted to radial
functions.  For $q\in\Lpwt$ define the quadratic form
\begin{equation}\label{eq:quadratic-form}
        a_q[y]=\int_0^R |y'|^2w\dd r+\int_0^R qy^2w\dd r,
        \qquad y\in H^1_{w,R}(0,R).
\end{equation}
The associated eigenvalue equation is \eqref{eq:radial-sl-main}, with the radial
regularity condition at the origin and the Dirichlet condition at $R$.
\begin{lemma}[KLMN theorem {\cite[Theorem~6.24]{Teschl2014}}]
Let $\mathcal H$ be a Hilbert space and let
\[
        \mathfrak a:\mathcal Q(\mathfrak a)\times \mathcal Q(\mathfrak a)
        \to \mathbb C
\]
be a semi-bounded closed quadratic form with form domain
$\mathcal Q(\mathfrak a)\subset \mathcal H$. Assume that
\[
        \mathfrak a[u]:=\mathfrak a(u,u)
        \geq
        \gamma\|u\|_{\mathcal H}^2,
        \qquad u\in\mathcal Q(\mathfrak a),
\]
for some $\gamma\in\mathbb R$.

Let
\[
        \mathfrak v:\mathcal Q(\mathfrak a)\times \mathcal Q(\mathfrak a)
        \to \mathbb C
\]
be another quadratic form which is relatively form-bounded with respect to
$\mathfrak a$ with relative bound less than one; namely, there exist constants
\[
        \vartheta\in[0,1),
        \qquad
        \varkappa\geq0,
\]
such that
\[
        |\mathfrak v[u]|
        \leq
        \vartheta\,\mathfrak a[u]
        +
        \varkappa\|u\|_{\mathcal H}^2,
        \qquad u\in\mathcal Q(\mathfrak a).
\]
Then the perturbed form
\[
        \mathfrak a+\mathfrak v
\]
defined on $\mathcal Q(\mathfrak a)$ by
\[
        (\mathfrak a+\mathfrak v)(u,v)
        =
        \mathfrak a(u,v)+\mathfrak v(u,v),
        \qquad u,v\in\mathcal Q(\mathfrak a),
\]
is closed and semi-bounded. More precisely,
\[
        (\mathfrak a+\mathfrak v)[u]
        \geq
        \bigl((1-\vartheta)\gamma-\varkappa\bigr)
        \|u\|_{\mathcal H}^2,
        \qquad u\in\mathcal Q(\mathfrak a).
\]
Consequently, the form $\mathfrak a+\mathfrak v$ gives rise to a
semi-bounded self-adjoint operator in $\mathcal H$. Here
\[
        \mathfrak v[u]=\mathfrak v(u,u).
\]
The constant $\gamma$ is a lower bound of $\mathfrak a$, the number
$\vartheta<1$ is the relative form bound of $\mathfrak v$ with respect to
$\mathfrak a$, and $\varkappa$ is an auxiliary constant in the relative
form-bound estimate.
\end{lemma}

\begin{lemma}\label{lem:form-bounded}
Assume $p>\frac{n}{2}$.  For each real-valued $q\in\Lpwt$ the form
$a_q$ is closed and bounded from below on $H^1_{w,R}(0,R)$.  The corresponding
self-adjoint operator in $\Ltwo$ has compact resolvent.
\end{lemma}

\begin{proof}
Let $y\in H^1_{w,R}(0,R)$ and let $\tilde y(x)=y(|x|)$,
$\tilde q(x)=q(|x|)$.  Up to the harmless factor $|\mathbb S^{n-1}|$, the
weighted integrals on $(0,R)$ are the radial integrals on $B_R$.  By H\"older's
inequality,
\[
 \int_0^R |q|y^2w\dd r\le \norm{q}_{p,w}\norm{y}_{2{\pd},w}^2,
 \qquad \pd=\frac{p}{p-1}.
\]
Since $p>n/2$, the exponent $2\pd$ is subcritical for the Sobolev embedding of
$H^1_0(B_R)$ when $n\ge3$; for $n=2$ it is finite and hence admissible.
{Using the identification $\tilde y(x)=y(|x|)$ and the polar-coordinate
formula, this embedding restricts to the radial weighted estimate}
\[
        \norm{y}_{2\pd,w}^2\le C\norm{y}_{H^1_{w,R}}^2 .
\]

{We now prove the infinitesimal form bound. For $M>0$, set}
\[
        {q_M(r)=\min\{|q(r)|,M\},\qquad
        r_M(r)=|q(r)|-q_M(r).}
\]
{Then $0\le q_M\le M$ and $\norm{r_M}_{p,w}\to0$ as
$M\to\infty$. Hence, by the preceding H\"older--Sobolev estimate,}
\[
\begin{aligned}
 {\int_0^R |q|y^2w\dd r}
 &{\le
 \int_0^R q_My^2w\dd r+\int_0^R r_My^2w\dd r}  \\
 &{\le
 M\int_0^R y^2w\dd r+\norm{r_M}_{p,w}\norm{y}_{2p',w}^2} \\
 &{\le
 M\int_0^R y^2w\dd r
 +C\norm{r_M}_{p,w}
 \left(
 \int_0^R |y'|^2w\dd r+\int_0^R y^2w\dd r
 \right).}
\end{aligned}
\]
{Given $\varepsilon>0$, choose $M$ so large that
$C\norm{r_M}_{p,w}\le\varepsilon$. Then there exists a constant
$C_\varepsilon(q)>0$ such that}
\[
 \int_0^R |q|y^2w\dd r
 \le \varepsilon\int_0^R |y'|^2w\dd r+C_\varepsilon(q)\int_0^R y^2w\dd r.
\]
Therefore the potential term is infinitesimally form-bounded with respect to the
Dirichlet form.  The KLMN theorem yields a closed lower-bounded form and hence a
self-adjoint operator.  Finally, the compact embedding
$H^1_0(B_R)\hookrightarrow L^2(B_R)$ restricts to the compact embedding
$H^1_{w,R}(0,R)\hookrightarrow L^2_w(0,R)$, so the resolvent is compact.
\end{proof}

\begin{proposition}\label{prop:radial-oscillation}
Let $q\in\Lpwt$ be real-valued.  The radial eigenvalues may be ordered as
\[
        \lambda_1(q)<\lambda_2(q)<\cdots<\lambda_m(q)<\cdots,
        \qquad \lambda_m(q)\to+\infty.
\]
Each radial eigenvalue is simple.  If $E_m(\cdot;q)$ is a real eigenfunction for
$\lambda_m(q)$, then $E_m$ has exactly $m-1$ zeros in $(0,R)$, and all of them
are simple.
\end{proposition}

\begin{proof}
By Lemma \ref{lem:form-bounded}, the spectrum is real, discrete, and consists of
finite-multiplicity eigenvalues tending to $+\infty$.  The simplicity is the
one-dimensional simplicity of the regular radial initial value problem.  Indeed,
{the space of solutions of \eqref{eq:radial-sl-main} satisfying
the regular radial condition at the origin is one-dimensional for each fixed
spectral parameter.}
If two linearly independent regular solutions satisfied the Dirichlet condition
at $R$, their weighted Wronskian
\[
        W(r)=w(r)(y_1y_2'-y_1'y_2)
\]
would be constant on every compact subinterval of $(0,R)$.
{Since both solutions satisfy $y_1(R)=y_2(R)=0$, this constant
has limiting value zero at $r=R$. Hence $W\equiv0$, and the two solutions are
linearly dependent, a contradiction.}
The oscillation count follows from the Sturm oscillation theorem for singular
Sturm--Liouville problems with a regular radial endpoint.
{Equivalently, one may obtain it by a standard regular
approximation on $[\delta,R]$ and then let $\delta\rightarrow0$.}
Passing to the limit $\delta\rightarrow0$ preserves the zeros because local
solutions are continuous and a non-trivial solution cannot have a double
interior zero.  Thus the $m$-th radial eigenfunction has exactly $m-1$ simple
interior zeros.
\end{proof}

We normalize eigenfunctions by
\begin{equation}\label{eq:eigen-normalization}
        \int_0^R E_m(r;q)^2w(r)\dd r=1.
\end{equation}
The zeros are denoted by
\[
        0<T_{1,m}(q)<\cdots<T_{m-1,m}(q)<R.
\]

\begin{lemma}\label{lem:weak-continuity}
Let $q_k\rightharpoonup q$ weakly in $\Lpwt$ and assume that $\{q_k\}$ is bounded
in $\Lpwt$.  Then, for each fixed $m$, $\lambda_m(q_k)\to\lambda_m(q)$.  After a
consistent choice of signs, $E_m(\cdot;q_k)\to E_m(\cdot;q)$ strongly in
$L^2_w(0,R)$ and in $C^1_{\mathrm{loc}}((0,R))$.  Consequently,
\[
        T_{i,m}(q_k)\to T_{i,m}(q),\qquad i=1,\ldots,m-1.
\]
\end{lemma}
\begin{proof}
It is enough to prove that every subsequence contains a further subsequence
converging to the asserted limit.

By the min--max characterization and the uniform form bound in Lemma
\ref{lem:form-bounded}, the sequence $\{\lambda_m(q_k)\}$ is bounded in
$\mathbb R$. Hence, after passing to a subsequence if necessary, we may assume
that
\begin{equation*}
        \lambda_m(q_k)\to \mu
\end{equation*}
for some $\mu\in\mathbb R$.

Let $E_k=E_m(\cdot;q_k)$ be normalized as in
\eqref{eq:eigen-normalization}. The eigenvalue equation and the same uniform
form bound imply that $\{E_k\}$ is bounded in $H^1_{w,R}(0,R)$. Indeed, since
$\{q_k\}$ is bounded in $L^p_w(0,R)$, the infinitesimal form bound is uniform
on this sequence: for each $\varepsilon\in(0,1)$ there exists
$C_\varepsilon>0$, independent of $k$, such that
\begin{equation*}
        \int_0^R |q_k|E_k^2w\,\dd r
        \le
        \varepsilon\int_0^R |E_k'|^2w\,\dd r
        +
        C_\varepsilon\int_0^R E_k^2w\,\dd r .
\end{equation*}
Since $\|E_k\|_{2,w}=1$ and $\{\lambda_m(q_k)\}$ is bounded, this gives a
uniform bound for
\begin{equation*}
        \int_0^R |E_k'|^2w\,\dd r .
\end{equation*}
Thus, up to a subsequence,
\begin{equation*}
        E_k\rightharpoonup E
        \quad \text{weakly in } H^1_{w,R}(0,R),
        \qquad
        E_k\to E
        \quad \text{strongly in } L^2_w(0,R).
\end{equation*}
In particular, $\|E\|_{2,w}=1$, and hence $E\not\equiv0$.

For every test function $\varphi\in H^1_{w,R}(0,R)$, the compact Sobolev
embedding used in Lemma \ref{lem:form-bounded} gives
\begin{equation*}
        E_k\to E
        \quad \text{strongly in } L^{2p^\dagger}_w(0,R).
\end{equation*}
Since
\begin{equation*}
        \varphi\in H^1_{w,R}(0,R)
        \hookrightarrow L^{2p^\dagger}_w(0,R),
\end{equation*}
H\"older's inequality yields
\begin{equation*}
\begin{aligned}
        \|E_k\varphi-E\varphi\|_{p^\dagger,w}
        &=
        \|(E_k-E)\varphi\|_{p^\dagger,w}  \\
        &\le
        \|E_k-E\|_{2p^\dagger,w}
        \|\varphi\|_{2p^\dagger,w}
        \to 0 .
\end{aligned}
\end{equation*}
Therefore
\begin{equation*}
        E_k\varphi\to E\varphi
        \quad \text{strongly in } L^{p^\dagger}_w(0,R).
\end{equation*}

Since $q_k\rightharpoonup q$ weakly in $L^p_w(0,R)$, we may pass to the limit
in the weak formulation
\begin{equation*}
        \int_0^R E_k'\varphi' w\,\dd r
        +
        \int_0^R q_kE_k\varphi w\,\dd r
        =
        \lambda_m(q_k)\int_0^R E_k\varphi w\,\dd r .
\end{equation*}
Indeed,
\begin{equation*}
\begin{aligned}
&\int_0^R q_kE_k\varphi w\,\dd r
-
\int_0^R qE\varphi w\,\dd r        \\
&\quad =
\int_0^R q_k(E_k\varphi-E\varphi)w\,\dd r
+
\int_0^R (q_k-q)E\varphi w\,\dd r
\to0,
\end{aligned}
\end{equation*}
because $\{q_k\}$ is bounded in $L^p_w(0,R)$ and
$E_k\varphi\to E\varphi$ strongly in $L^{p^\dagger}_w(0,R)$. Hence
\begin{equation*}
        \int_0^R E'\varphi' w\,\dd r
        +
        \int_0^R qE\varphi w\,\dd r
        =
        \mu\int_0^R E\varphi w\,\dd r .
\end{equation*}
Thus $E$ is a nontrivial eigenfunction of the limit operator.

We now identify the eigenvalue index. First, by testing the min--max formula
on the fixed $m$-dimensional space
\begin{equation*}
        V_m(q)
        =
        \operatorname{span}\{E_1(\cdot;q),\ldots,E_m(\cdot;q)\},
\end{equation*}
and using the weak convergence of $q_k$ on this finite-dimensional space, one
gets
\begin{equation*}
        \limsup_{k\to\infty}\lambda_m(q_k)\le \lambda_m(q).
\end{equation*}
Conversely, suppose along a subsequence that
\begin{equation*}
        \lambda_m(q_k)\to \mu<\lambda_m(q).
\end{equation*}
For each $\ell=1,\ldots,m$, choose normalized eigenfunctions
\begin{equation*}
        E_{\ell,k}=E_\ell(\cdot;q_k).
\end{equation*}
Since
\begin{equation*}
        \lambda_\ell(q_k)\le \lambda_m(q_k),
\end{equation*}
the same compactness argument applied to $\{E_{\ell,k}\}$ gives, after a
diagonal subsequence, nonzero mutually orthogonal limits
$E_1,\ldots,E_m$, each of which is an eigenfunction of the limit operator with
eigenvalue not larger than $\mu$. This gives at least $m$ linearly independent
eigenfunctions of the limit operator with eigenvalues strictly below
$\lambda_m(q)$, contradicting the definition of the $m$-th eigenvalue. Hence
\begin{equation*}
        \liminf_{k\to\infty}\lambda_m(q_k)\ge \lambda_m(q).
\end{equation*}
Therefore
\begin{equation*}
        \mu=\lambda_m(q).
\end{equation*}
Since $\lambda_m(q)$ is simple by Proposition \ref{prop:radial-oscillation},
the limit eigenfunction must be $\pm E_m(\cdot;q)$. After choosing the signs
consistently, the limit is $E_m(\cdot;q)$. Since every subsequence admits a
further subsequence with this same limit, the whole sequence satisfies
\begin{equation*}
        \lambda_m(q_k)\to\lambda_m(q),
        \qquad
        E_m(\cdot;q_k)\to E_m(\cdot;q)
        \quad \text{strongly in } L^2_w(0,R).
\end{equation*}

It remains to prove the local $C^1$ convergence. Let
\begin{equation*}
        K\Subset K_1\Subset(0,R).
\end{equation*}
On $K_1$, the weight $w(r)=r^{n-1}$ is bounded above and below away from zero.
The equation may be written in divergence form as
\begin{equation*}
        -(wE_k')'=w(\lambda_m(q_k)-q_k)E_k .
\end{equation*}
The already obtained $H^1$ bound, together with the one-dimensional embedding
\begin{equation*}
        H^1(K_1)\hookrightarrow L^\infty(K_1),
\end{equation*}
implies that $\{E_k\}$ is bounded in $L^\infty(K_1)$. Since $\{q_k\}$ is
bounded in $L^p(K_1)$ and $\{\lambda_m(q_k)\}$ is bounded, the right-hand side
\begin{equation*}
        w(\lambda_m(q_k)-q_k)E_k
\end{equation*}
is bounded in $L^p(K_1)$. By one-dimensional local elliptic regularity,
$\{E_k\}$ is bounded in $W^{2,p}(K)$. Since $p>1$, the embedding
\begin{equation*}
        W^{2,p}(K)\hookrightarrow C^1(K)
\end{equation*}
is compact. The local limit has already been identified as $E_m(\cdot;q)$.
Hence
\begin{equation*}
        E_m(\cdot;q_k)\to E_m(\cdot;q)
        \quad \text{in } C^1(K).
\end{equation*}
Since $K\Subset(0,R)$ is arbitrary, we obtain
\begin{equation*}
        E_m(\cdot;q_k)\to E_m(\cdot;q)
        \quad \text{in } C^1_{\mathrm{loc}}((0,R)).
\end{equation*}

Finally, the zeros of $E_m(\cdot;q)$ are simple. If $T_{i,m}(q)$ is a zero of
$E_m(\cdot;q)$, then
\begin{equation*}
        E_m'(T_{i,m}(q);q)\neq0.
\end{equation*}
Therefore, for all sufficiently large $k$, there exists a unique zero of
$E_m(\cdot;q_k)$ near $T_{i,m}(q)$, and this zero converges to $T_{i,m}(q)$.
Since the $m$-th radial eigenfunction has exactly $m-1$ simple zeros, these
zeros are precisely $T_{i,m}(q_k)$, $i=1,\ldots,m-1$. Consequently,
\begin{equation*}
        T_{i,m}(q_k)\to T_{i,m}(q),
        \qquad i=1,\ldots,m-1.
\end{equation*}
This completes the proof.
\end{proof}
\begin{proposition}\label{prop:nodal-C1}
For each admissible pair $(i,m)$, the map
\[
        q\mapsto T_{i,m}(q)
\]
is locally $C^1$ from $\Lpwt$ to $(0,R)$.
\end{proposition}
\begin{proof}
Fix $q_*\in L^p_w(0,R)$ and put
\begin{equation*}
        E_*:=E_m(\cdot;q_*),
        \qquad
        T_*:=T_{i,m}(q_*).
\end{equation*}
Choose an open interval $I\Subset(0,R)$ such that $T_*\in I$.

By Proposition \ref{prop:radial-oscillation}, the eigenvalue
$\lambda_m(q_*)$ is simple. Moreover, by Lemma \ref{lem:form-bounded}, for
each $h\in L^p_w(0,R)$ the bilinear form
\begin{equation*}
        (y,z)\mapsto \int_0^R h(r)y(r)z(r)w(r)\,\dd r
\end{equation*}
is infinitesimally form-bounded with respect to the Dirichlet form. Hence
$q\mapsto a_q$ is a locally $C^1$ family of closed lower-bounded forms in the
form topology. The standard perturbation theorem for simple eigenvalues of
closed forms therefore implies that, after imposing the normalization
\eqref{eq:eigen-normalization} and fixing the sign by
\begin{equation*}
        \langle E_m(\cdot;q),E_*\rangle_w>0,
\end{equation*}
the maps
\begin{equation*}
        q\mapsto \lambda_m(q)\in\mathbb R,
        \qquad
        q\mapsto E_m(\cdot;q)\in H^1_{w,R}(0,R)
\end{equation*}
are locally $C^1$ in a neighborhood of $q_*$. See
\cite[Chapter~VII]{Kato1995}.

We next upgrade the dependence of the eigenfunction to local $C^1$ regularity
in the spatial variable. On $I$, the weight $w(r)=r^{n-1}$ is bounded above
and below away from zero. The eigenvalue equation can be written locally as
\begin{equation*}
        -(wE_m')'
        =
        w(\lambda_m(q)-q)E_m .
\end{equation*}
Using the already obtained $C^1$ dependence of
$q\mapsto \lambda_m(q)$ and $q\mapsto E_m(\cdot;q)$ in the energy topology,
together with the local one-dimensional elliptic regularity, we get
\begin{equation*}
        q\mapsto E_m(\cdot;q)
\end{equation*}
locally $C^1$ from $L^p_w(0,R)$ into $W^{2,p}(I)$. Since $p>1$ and
\begin{equation*}
        W^{2,p}(I)\hookrightarrow C^1(\overline I),
\end{equation*}
it follows that
\begin{equation*}
        q\mapsto E_m(\cdot;q)
\end{equation*}
is locally $C^1$ from $L^p_w(0,R)$ into $C^1(\overline I)$.

Now define
\begin{equation*}
        F(q,T):=E_m(T;q),
        \qquad
        (q,T)\in U\times I,
\end{equation*}
where $U$ is a sufficiently small neighborhood of $q_*$ in $L^p_w(0,R)$.
By the preceding paragraph,
\begin{equation*}
        F:U\times I\to\mathbb R
\end{equation*}
is a $C^1$ map. Moreover,
\begin{equation*}
        F(q_*,T_*)=E_*(T_*)=0.
\end{equation*}
Since the zero $T_*$ is simple, Proposition \ref{prop:radial-oscillation}
gives
\begin{equation*}
        \partial_TF(q_*,T_*)
        =
        E_*'(T_*)\neq0.
\end{equation*}
Therefore the Banach-space implicit function theorem yields a neighborhood
$U_0\subset U$ of $q_*$ and a unique $C^1$ map
\begin{equation*}
        T:U_0\to I
\end{equation*}
such that
\begin{equation*}
        T(q_*)=T_*,
        \qquad
        E_m(T(q);q)=0,
        \qquad q\in U_0.
\end{equation*}
Since the interior zeros of the $m$-th radial eigenfunction are simple and
ordered, this locally defined zero is precisely the continuation of the
$i$-th nodal point. Hence
\begin{equation*}
        T(q)=T_{i,m}(q)
\end{equation*}
for all $q$ sufficiently close to $q_*$. Consequently,
\begin{equation*}
        q\mapsto T_{i,m}(q)
\end{equation*}
is locally $C^1$ at $q_*$. Since $q_*$ was arbitrary, the proof is complete.

\end{proof}
\begin{theorem}\label{thm:spectral}
Under $p>\frac{n}{2}$, the radial problem \eqref{eq:radial-sl-main}
defines a self-adjoint operator in $\Ltwo$ with compact resolvent; its radial
eigenvalues are simple, its radial eigenfunctions have the nodal structure
described in Proposition \ref{prop:radial-oscillation}, and the nodal maps are
locally $C^1$ and weakly sequentially continuous on bounded subsets of $\Lpwt$.
\end{theorem}

\begin{proof}
This is the combination of Lemma \ref{lem:form-bounded}, Proposition
\ref{prop:radial-oscillation}, Lemma \ref{lem:weak-continuity}, and Proposition
\ref{prop:nodal-C1}.
\end{proof}

We next compute the derivative of a nodal map.  Fix $q\in\Lpwt$, $m\ge2$, and
write
\[
        E(r)=E_m(r;q),\qquad \lambda=\lambda_m(q),\qquad T_j=T_{j,m}(q).
\]
Set
\[
        A_j(q)=\int_0^{T_j}E(r)^2w(r)\dd r,
        \qquad
        B_j(q)=\int_{T_j}^R E(r)^2w(r)\dd r.
\]
Then $A_j+B_j=1$.  Define
\begin{equation}\label{eq:Hj-def}
H_j(r;q)=
\begin{cases}
\displaystyle \frac{B_j(q)}{w(T_j)E'(T_j)^2},&0<r<T_j,\\[1.2ex]
\displaystyle -\frac{A_j(q)}{w(T_j)E'(T_j)^2},&T_j<r<R.
\end{cases}
\end{equation}
{The value of $H_j$ at $r=T_j$ is immaterial and may be assigned
arbitrarily, since $T_j$ is a single point and all subsequent formulas use
$H_j$ only under weighted integration or through its one-sided limits.}

\begin{lemma}
\label{lem:eigenvalue-derivative}
Let $q\in L^p_w(0,R)$ be real-valued and let $h\in L^p_w(0,R)$.  Set
\[
        q_s=q+sh,
        \qquad
        \lambda_s=\lambda_m(q_s),
        \qquad
        E_s=E_m(\cdot;q_s),
\]
where $E_s$ is normalized by
\[
        \int_0^R E_s(r)^2w(r)\,\dd r=1
\]
and chosen with a continuous sign.  Then
\begin{equation}\label{eq:eigenvalue-derivative}
        \left.\frac{\dd}{\dd s}\lambda_m(q+sh)\right|_{s=0}
        =
        \int_0^R h(r)E_m(r;q)^2w(r)\,\dd r .
\end{equation}
\end{lemma}

\begin{proof}
Put
\[
        E=E_m(\cdot;q),
        \qquad
        \lambda=\lambda_m(q),
        \qquad
        z=\left.\frac{\dd}{\dd s}E_s\right|_{s=0},
        \qquad
        \lambda'=\left.\frac{\dd}{\dd s}\lambda_s\right|_{s=0}.
\]
By Lemma \ref{lem:form-bounded} and the standard perturbation theorem for
simple eigenvalues of closed forms, the above derivatives are well defined.
Differentiating the normalization at $s=0$ gives
\begin{equation}\label{eq:normalization-derivative}
        \int_0^R E(r)z(r)w(r)\,\dd r=0 .
\end{equation}

Since $E_s$ is $L^2_w$-normalized, the Rayleigh identity gives
\[
        \lambda_s
        =
        \int_0^R |E_s'(r)|^2w(r)\,\dd r
        +
        \int_0^R q_s(r)E_s(r)^2w(r)\,\dd r .
\]
Differentiating this identity at $s=0$ yields
\begin{align}
        \lambda'
        &=
        2\int_0^R E'(r)z'(r)w(r)\,\dd r
        +
        \int_0^R h(r)E(r)^2w(r)\,\dd r        \notag\\
        &\quad
        +
        2\int_0^R q(r)E(r)z(r)w(r)\,\dd r .
        \label{eq:rayleigh-differentiated}
\end{align}
On the other hand, testing
\[
        -w^{-1}(wE')'+qE=\lambda E
\]
by $z$ gives
\begin{equation}\label{eq:test-E-by-z}
        \int_0^R E'(r)z'(r)w(r)\,\dd r
        +
        \int_0^R q(r)E(r)z(r)w(r)\,\dd r
        =
        \lambda\int_0^R E(r)z(r)w(r)\,\dd r .
\end{equation}
By \eqref{eq:normalization-derivative}, the right-hand side of
\eqref{eq:test-E-by-z} is zero. Hence the two terms containing $z$ in
\eqref{eq:rayleigh-differentiated} cancel, and we obtain
\[
        \lambda'
        =
        \int_0^R h(r)E(r)^2w(r)\,\dd r .
\]
This proves \eqref{eq:eigenvalue-derivative}.
\end{proof}

\begin{theorem}
\label{thm:derivative}
For every $h\in\Lpwt$,
\begin{equation}\label{eq:nodal-derivative}
        D_qT_{j,m}(q)[h]
        =
        \int_0^R H_j(r;q)E_m(r;q)^2h(r)w(r)\,\dd r .
\end{equation}
\end{theorem}

\begin{proof}
Let
\[
        q_s=q+sh,
        \qquad
        E_s=E_m(\cdot;q_s),
        \qquad
        \lambda_s=\lambda_m(q_s),
        \qquad
        T_s=T_{j,m}(q_s).
\]
The eigenfunctions are normalized in $L^2_w(0,R)$ and chosen with a continuous
sign.  Put
\[
        E=E_m(\cdot;q),
        \qquad
        \lambda=\lambda_m(q),
        \qquad
        T_j=T_{j,m}(q),
\]
and
\[
        z=\left.\frac{\dd}{\dd s}E_s\right|_{s=0},
        \qquad
        \lambda'=\left.\frac{\dd}{\dd s}\lambda_s\right|_{s=0}.
\]
Differentiating the eigenvalue equation gives, in the weak sense and hence
locally classically away from the origin,
\begin{equation}\label{eq:z-equation}
        -(wz')'+w(q-\lambda)z
        =
        w(\lambda'-h)E .
\end{equation}
The unperturbed equation is
\begin{equation}\label{eq:E-equation}
        -(wE')'+w(q-\lambda)E=0 .
\end{equation}
Multiplying \eqref{eq:z-equation} by $E$, multiplying
\eqref{eq:E-equation} by $z$, and subtracting yields
\begin{equation}\label{eq:wronskian}
        \frac{\dd}{\dd r}
        \left[
        w(r)\bigl(z(r)E'(r)-E(r)z'(r)\bigr)
        \right]
        =
        w(r)(\lambda'-h(r))E(r)^2 .
\end{equation}
The boundary term at the origin vanishes in the sense that
\[
        \lim_{r\to0}
        w(r)\bigl(z(r)E'(r)-E(r)z'(r)\bigr)
        =
        0 .
\]
Indeed, regular radial solutions are bounded near the origin and have finite
radial energy, while $w(r)=r^{n-1}\to0$.  Equivalently, one may integrate over
$(\delta,T_j)$ and then let $\delta\downarrow0$.

Integrating \eqref{eq:wronskian} over $(0,T_j)$ and using $E(T_j)=0$, we get
\[
        w(T_j)z(T_j)E'(T_j)
        =
        \int_0^{T_j}(\lambda'-h)E^2w\,\dd r .
\]
Since $E_s(T_s)=0$, differentiation at $s=0$ gives
\[
        z(T_j)+E'(T_j)D_qT_{j,m}(q)[h]=0 .
\]
Therefore
\[
        D_qT_{j,m}(q)[h]
        =
        \frac{
        \displaystyle
        \int_0^{T_j}(h-\lambda')E^2w\,\dd r
        }{
        w(T_j)E'(T_j)^2
        } .
\]
By Lemma \ref{lem:eigenvalue-derivative},
\[
        \lambda'
        =
        \int_0^R hE^2w\,\dd r .
\]
Using
\[
        A_j=\int_0^{T_j}E^2w\,\dd r,
        \qquad
        B_j=\int_{T_j}^R E^2w\,\dd r,
        \qquad
        A_j+B_j=1,
\]
we obtain
\[
\begin{aligned}
        D_qT_{j,m}(q)[h]
        &=
        \frac{
        \displaystyle
        B_j\int_0^{T_j}hE^2w\,\dd r
        -
        A_j\int_{T_j}^RhE^2w\,\dd r
        }{
        w(T_j)E'(T_j)^2
        } .
\end{aligned}
\]
By the definition of $H_j$, this is exactly \eqref{eq:nodal-derivative}.
\end{proof}
\begin{corollary}\label{cor:translation-diff}
For every $j=1,\ldots,m-1$,
\[
        \int_0^R H_j(r;q)E_m(r;q)^2w(r)\dd r=0.
\]
Consequently $D_qT_{j,m}(q)[1]=0$, {which is the infinitesimal form of the global translation invariance}
$T_{j,m}(q+a)=T_{j,m}(q)$ {for every constant $a\in\mathbb R$}.
\end{corollary}

\begin{proof}
Using $A_j+B_j=1$ and the definition of $H_j$, the integral equals
\[
        \frac{B_jA_j-A_jB_j}{w(T_j)E'(T_j)^2}=0.
\]
\end{proof}

\section{Realization and geometry of finite nodal constraints}\label{sec:constraints-geometry}

We first show that finite radial nodal constraints are always admissible.  This
is a structural point: the optimization problems above are not built on an
implicit non-emptiness assumption.

\begin{theorem}\label{thm:nonempty-full}
For every $m\ge2$ and every $\tau=(\tau_1,\ldots,\tau_{m-1})\in\Omega_m$, there
exists a bounded piecewise constant radial potential $q_\tau$ such that
$T_m(q_\tau)=\tau$.  Consequently $\Sset^{(m)}_\tau\ne\emptyset$ for every
$\tau\in\Omega_m$, and $\Sset_{i,m}(T^*)\ne\emptyset$ for every $T^*\in(0,R)$.
\end{theorem}
\begin{proof}
Set
\[
        \tau_0=0,\qquad \tau_m=R,
        \qquad
        I_k=(\tau_{k-1},\tau_k),\quad k=1,\ldots,m .
\]
On $I_1$, let $\phi_1>0$ be the first eigenfunction of
\[
        -w^{-1}(w\phi_1')'=\mu_1\phi_1,
        \qquad
        \phi_1'(0)=0,
        \qquad
        \phi_1(\tau_1)=0 .
\]
Here $\phi_1'(0)=0$ is understood as the natural radial regularity condition at
the origin.  The first eigenfunction can be chosen strictly positive in
$(0,\tau_1)$, and the boundary point lemma gives
\[
        \phi_1'(\tau_1^-)<0 .
\]

For $k=2,\ldots,m$, let $\phi_k>0$ be the first eigenfunction of the regular
Dirichlet problem
\[
        -w^{-1}(w\phi_k')'=\mu_k\phi_k,
        \qquad
        \phi_k(\tau_{k-1})=\phi_k(\tau_k)=0 .
\]
Since $\phi_k>0$ in $I_k$, the boundary point lemma gives
\[
        \phi_k'(\tau_{k-1}^+)>0,
        \qquad
        \phi_k'(\tau_k^-)<0,
        \qquad k=2,\ldots,m .
\]

Fix $\lambda_*\in\mathbb R$ and define
\[
        q_\tau(r)=\lambda_*-\mu_k,
        \qquad r\in I_k .
\]
Since there are only finitely many intervals and each $\mu_k$ is finite,
$q_\tau$ is a bounded piecewise constant function; in particular,
\[
        q_\tau\in L^\infty(0,R)\subset L^p_w(0,R).
\]
Moreover, on each interval $I_k$,
\[
        -w^{-1}(w\phi_k')'+q_\tau\phi_k
        =
        \mu_k\phi_k+(\lambda_*-\mu_k)\phi_k
        =
        \lambda_*\phi_k .
\]

We now glue the eigenfunctions with alternating signs.  Choose $a_1>0$.  Having
chosen $a_k>0$, choose $a_{k+1}>0$ so that the weighted fluxes match at
$\tau_k$:
\begin{equation}\label{eq:flux-gluing}
        (-1)^{k-1}a_kw(\tau_k)\phi_k'(\tau_k^-)
        =
        (-1)^ka_{k+1}w(\tau_k)\phi_{k+1}'(\tau_k^+),
        \qquad k=1,\ldots,m-1 .
\end{equation}
This is possible because
\[
        \phi_k'(\tau_k^-)<0,
        \qquad
        \phi_{k+1}'(\tau_k^+)>0 .
\]
Indeed, since $w(\tau_k)>0$, \eqref{eq:flux-gluing} is equivalent to
\[
        a_{k+1}
        =
        a_k\frac{-\phi_k'(\tau_k^-)}{\phi_{k+1}'(\tau_k^+)}
        >0 .
\]
Thus the constants $a_k$ are determined recursively.

Define
\[
        y(r)=(-1)^{k-1}a_k\phi_k(r),
        \qquad r\in I_k .
\]
Then $y$ is continuous, because both adjacent pieces vanish at each interface
$\tau_k$.  Moreover, \eqref{eq:flux-gluing} gives
\[
        (wy')(\tau_k^-)=(wy')(\tau_k^+),
        \qquad k=1,\ldots,m-1 .
\]
Therefore no Dirac mass is produced by differentiating $wy'$ across the
interfaces.

Equivalently, for every test function $\psi\in H^1_{w,R}(0,R)$, integration by
parts on the subintervals $I_k$ gives
\[
        \sum_{k=1}^m\int_{I_k} y'\psi' w\,\dd r
        +
        \sum_{k=1}^m\int_{I_k} q_\tau y\psi w\,\dd r
        =
        \lambda_*\sum_{k=1}^m\int_{I_k} y\psi w\,\dd r .
\]
Here all interior boundary terms cancel by \eqref{eq:flux-gluing}; the endpoint
contribution at $r=0$ vanishes by the radial regularity condition, and the
Dirichlet endpoint condition at $R$ is encoded in $y(R)=0$.

Thus $y$ is a weak solution of
\[
        -w^{-1}(wy')'+q_\tau y=\lambda_*y,
        \qquad
        y'(0)=0,
        \qquad
        y(R)=0 .
\]
Since each $\phi_k$ is strictly positive in $I_k$, the glued function $y$ has no
zeros inside any $I_k$.  Hence its only interior zeros are precisely
\[
        \tau_1,\ldots,\tau_{m-1}.
\]
Therefore $y$ has exactly $m-1$ interior zeros.  By the radial
Sturm--Liouville oscillation theorem, $y$ corresponds to the $m$-th radial
eigenvalue of the potential $q_\tau$. Consequently,
\[
        \lambda_*=\lambda_m(q_\tau),
        \qquad
        T_m(q_\tau)=\tau .
\]
This proves $S_\tau^{(m)}\neq\varnothing$.

It remains to prove the single-node statement.  Let $T^*\in(0,R)$ and fix an
admissible pair $(i,m)$.  Choose $i-1$ points in $(0,T^*)$ and
$m-1-i$ points in $(T^*,R)$, and set $\tau_i=T^*$.  Then
$\tau=(\tau_1,\ldots,\tau_{m-1})\in\Omega_m$.  By the construction above, there
exists a bounded piecewise constant potential $q_\tau$ such that
\[
        T_m(q_\tau)=\tau .
\]
In particular,
\[
        T_{i,m}(q_\tau)=T^* .
\]
Hence
\[
        S_{i,m}(T^*)\neq\varnothing .
\]
The proof is complete.
\end{proof}
\begin{lemma}\label{lem:independent}
For fixed $m\ge2$ and $q\in\Lpwt$, the linear functionals
\[
        D_qT_{1,m}(q),\ldots,D_qT_{m-1,m}(q)
\]
are linearly independent on $\Lpwt$.
\end{lemma}

\begin{proof}
Suppose $\sum_{j=1}^{m-1}c_jD_qT_{j,m}(q)=0$.  By Theorem
\ref{thm:derivative},
\[
 \int_0^R\left(\sum_{j=1}^{m-1}c_jH_j(r;q)\right)E_m(r;q)^2h(r)w(r)\dd r=0
\]
for every $h\in\Lpwt$.  Since $E_m^2>0$ away from its nodes, the step function
$\sum c_jH_j$ vanishes on every open nodal interval.  At the node $T_j$, the
function $H_j$ has the jump
\[
        H_j(T_j^+)-H_j(T_j^-)
        =-\frac{1}{w(T_j)E_m'(T_j)^2}\ne0,
\]
whereas $H_\ell$ has no jump at $T_j$ when $\ell\ne j$.  Taking jumps gives
$c_j=0$ for each $j$.
\end{proof}

\begin{theorem}\label{thm:manifold}
For every $m\ge2$ and $\tau\in\Omega_m$, the set $\Sset^{(m)}_\tau$ is a
$C^1$ submanifold of $\Lpwt$ of codimension $m-1$.  For every admissible
$(i,m,T^*)$, the single-node set $\Sset_{i,m}(T^*)$ is a $C^1$ hypersurface of
$\Lpwt$.
\end{theorem}

\begin{proof}
The full nodal map $T_m:\Lpwt\to\R^{m-1}$ is locally $C^1$ by Proposition
\ref{prop:nodal-C1}.  Lemma \ref{lem:independent} says that its differential has
rank $m-1$ at every point.  The Banach-space implicit function theorem therefore
implies that $T_m^{-1}(\tau)$ is a $C^1$ submanifold of codimension $m-1$.  The
one-node case is the same argument with a single nonzero derivative.
\end{proof}

{We now
record the non-uniqueness of potentials realizing prescribed nodal data before
any target potential \(q_0\) or variational selection criterion is imposed.}
Let
\[
        I=\{(i_\ell,m_\ell):\ell=1,\ldots,N\}
\]
be a finite family of admissible nodal indices and define
\[
        T_I(q)=\big(T_{i_1,m_1}(q),\ldots,T_{i_N,m_N}(q)\big).
\]
For a prescribed vector $\tau\in\R^N$, set
\[
        \Sset_I(\tau)=\{q\in\Lpwt:T_I(q)=\tau\}.
\]

\begin{theorem}\label{thm:global-nonunique}
Let $q\in\Lpwt$ and $a\in\R$.  Then
\[
        T_{i,m}(q+a)=T_{i,m}(q)
\]
for every admissible pair $(i,m)$.  Consequently, if $q\in\Sset_I(\tau)$, then
$q+a\in\Sset_I(\tau)$ for every $a\in\R$.  {Thus nodal radii alone cannot
distinguish a potential from any of its constant shifts; in particular, no
finite set of radial nodal data determines the potential globally.}
\end{theorem}

\begin{proof}
If $E_m(\cdot;q)$ is an eigenfunction associated with $\lambda_m(q)$, then
\[
        -w^{-1}(wE_m')'+qE_m=\lambda_m(q)E_m.
\]
Replacing $q$ by $q+a$ shifts the eigenvalue to $\lambda_m(q)+a$ and leaves the
eigenfunction unchanged.  Nodal points depend only on the zeros of the
eigenfunction, and hence all radial nodal points are unchanged.
\end{proof}

{Theorem \ref{thm:global-nonunique} is the global version of Corollary
\ref{cor:translation-diff}: the corollary records the same translation
invariance at the differential level, while the theorem states the exact
invariance for every constant shift of the potential.}

\begin{definition}
A point $q\in\Sset_I(\tau)$ is called regular for $T_I$ if
$DT_I(q):\Lpwt\to\R^N$ is onto.
\end{definition}

\begin{theorem}\label{thm:local-nonunique}
Let $q\in\Sset_I(\tau)$ be a regular point.  Then $\Sset_I(\tau)$ is, near $q$, a
$C^1$ submanifold of $\Lpwt$ of codimension $N$.  Its tangent space is
$\ker DT_I(q)$, which is infinite-dimensional.  In particular, every
neighborhood of $q$ contains infinitely many distinct potentials with the same
finite nodal data, and there are nonconstant perturbation directions preserving
the finite nodal data to first order.
\end{theorem}

\begin{proof}
The statement follows from the Banach-space implicit function theorem.  The
ambient space $\Lpwt$ is infinite-dimensional, whereas the range of $DT_I(q)$ is
finite-dimensional.  Hence the kernel is infinite-dimensional.  The constant
direction belongs to the kernel by Theorem \ref{thm:global-nonunique}, but an
infinite-dimensional kernel contains nonconstant directions as well.
\end{proof}

\begin{corollary}\label{cor:single-local-nonunique}
For every admissible pair $(i,m)$ and every $q\in\Lpwt$, the level set
\[
        \{\tilde q:T_{i,m}(\tilde q)=T_{i,m}(q)\}
\]
is locally an infinite-dimensional $C^1$ hypersurface near $q$.
\end{corollary}

\begin{proof}
The derivative $D_qT_{i,m}(q)$ is nonzero by Theorem \ref{thm:derivative}; hence
$q$ is a regular point for the one-component nodal map.  The conclusion follows
from Theorem \ref{thm:local-nonunique}.
\end{proof}

\section{Variational selection and radial Euler--Lagrange equations}\label{sec:variational}

\begin{theorem}\label{thm:existence}
Assume $p>\frac{n}{2}$ and let $q_0\in\Lpwt$.
\begin{enumerate}[label=\textup{(\roman*)}]
\item For every $m\ge2$ and $\tau\in\Omega_m$, the problem
\eqref{eq:main-full-problem} admits at least one minimizer.
\item For every admissible $(i,m,T^*)$, the problem \eqref{eq:main-single-problem}
admits at least one minimizer.
\end{enumerate}
\end{theorem}

\begin{proof}
We prove the full-vector case; the single-node case is identical.  The constraint
set is nonempty by Theorem \ref{thm:nonempty-full}.  Let $\{q_k\}\subset
\Sset^{(m)}_\tau$ be a minimizing sequence.  {Since $\{q_k\}$ is minimizing, $\|q_k-q_0\|_{p,w}$ is bounded; hence
$\{q_k\}$ is bounded in $L^p_w(0,R)$ }. Then $\{q_k\}$ is bounded in the
reflexive space $\Lpwt$, and hence, after passing to a subsequence,
$q_k\rightharpoonup \hat q$ weakly in $\Lpwt$.  By Lemma
\ref{lem:weak-continuity}, $T_m(\hat q)=\tau$, so $\hat q\in\Sset^{(m)}_\tau$.
Weak lower semicontinuity of the norm gives
\[
        \norm{\hat q-q_0}_{p,w}\le \liminf_{k\to\infty}\norm{q_k-q_0}_{p,w}.
\]
Thus $\hat q$ is a minimizer.
\end{proof}

We use
\[
        \phip(s)=|s|^{p-2}s,
        \qquad \phip^{-1}=\phid.
\]
Let $\hat q$ be a minimizer of the full-vector problem and assume
$\hat q\ne q_0$.  Write
\[
        \hat E=E_m(\cdot;\hat q),\qquad
        \hat\lambda=\lambda_m(\hat q),\qquad
        \hat T_j=T_{j,m}(\hat q)=\tau_j.
\]
Let $\hat H_j(r)=H_j(r;\hat q)$.

\begin{theorem}\label{thm:EL-full}
Let $\hat q\ne q_0$ be an optimal potential for \eqref{eq:main-full-problem}.
Then there exist multipliers $c_1,\ldots,c_{m-1}\in\R$ such that
\begin{equation}\label{eq:EL-full}
 \norm{\hat q-q_0}_{p,w}^{1-p}\phip(\hat q(r)-q_0(r))
 =\sum_{j=1}^{m-1}c_j\hat H_j(r)\hat E(r)^2
\end{equation}
for a.e. $r\in(0,R)$.  If
\[
        \Gamma(r)=\sum_{j=1}^{m-1}c_j\hat H_j(r)
\]
and $\Gamma_k$ is its constant value on $I_k=(\tau_{k-1},\tau_k)$, then
\begin{equation}\label{eq:q-representation-full}
 \hat q(r)-q_0(r)=
 \norm{\hat q-q_0}_{p,w}\,\operatorname{sgn}(\Gamma_k)|\Gamma_k|^{\pd-1}
 |\hat E(r)|^{2\pd-2},\qquad r\in I_k.
\end{equation}
If $\Gamma_k\ne0$ and
\[
        U_k=\left(\norm{\hat q-q_0}_{p,w}|\Gamma_k|^{\pd-1}\right)^{1/(2\pd-2)}\hat E
        \quad\hbox{on }I_k,
\]
then $U_k$ solves
\begin{equation}\label{eq:piecewise-full}
        -w^{-1}(wU_k')'+q_0(r)U_k+
        \operatorname{sgn}(\Gamma_k)|U_k|^{2\pd-2}U_k
        =\hat\lambda U_k,
        \qquad r\in I_k.
\end{equation}
If $\Gamma_k=0$, then $\hat E$ solves the corresponding linear equation with
potential $q_0$ on $I_k$.
\end{theorem}

\begin{proof}
The constraint set is a $C^1$ finite-codimensional manifold by Theorem
\ref{thm:manifold}, and the gradients of the constraints are independent by
Lemma \ref{lem:independent}.  Applying the Lagrange multiplier rule to
$q\mapsto\norm{q-q_0}_{p,w}$ gives \eqref{eq:EL-full}, since
\[
D\norm{q-q_0}_{p,w}[h]
=\norm{q-q_0}_{p,w}^{1-p}\int_0^R\phip(q-q_0)h\,w\dd r
\]
for $q\ne q_0$, and the constraint derivatives are given by Theorem
\ref{thm:derivative}.  Formula \eqref{eq:q-representation-full} follows by
applying $\phip^{-1}=\phid$.  Substitution into the eigenvalue equation for
$\hat E$ and the indicated rescaling yield \eqref{eq:piecewise-full}.
\end{proof}

For a single prescribed node, put $T^*=T_{i,m}(\hat q)$,
$\hat E=E_m(\cdot;\hat q)$, and
\[
        \hat A=\int_0^{T^*}\hat E^2w\dd r,
        \qquad
        \hat B=\int_{T^*}^R\hat E^2w\dd r.
\]

\begin{theorem}\label{thm:single-critical}
Let $\hat q\ne q_0$ be a minimizer of \eqref{eq:main-single-problem}.  Then there
are $\varepsilon\in\{-1,1\}$ and positive constants $a,b$ such that
\[
\hat q(r)-q_0(r)=
\begin{cases}
\varepsilon a|\hat E(r)|^{2\pd-2},&0<r<T^*,\\
-\varepsilon b|\hat E(r)|^{2\pd-2},&T^*<r<R.
\end{cases}
\]
Define
\[
U(r)=
\begin{cases}
a^{1/(2\pd-2)}\hat E(r),&0<r<T^*,\\
b^{1/(2\pd-2)}\hat E(r),&T^*<r<R.
\end{cases}
\]
Then $U$ satisfies
\begin{equation}\label{eq:single-left}
        -r^{1-n}(r^{n-1}U')'+q_0(r)U+
        \varepsilon |U|^{2\pd-2}U=\hat\lambda U,
        \qquad 0<r<T^*,
\end{equation}
and
\begin{equation}\label{eq:single-right}
        -r^{1-n}(r^{n-1}U')'+q_0(r)U-
        \varepsilon |U|^{2\pd-2}U=\hat\lambda U,
        \qquad T^*<r<R.
\end{equation}
Moreover,
\[
        U'(0)=0,
        \qquad U(T^*)=0,
        \qquad U(R)=0,
\]
$U$ has exactly $i-1$ zeros in $(0,T^*)$ and $m-1-i$ zeros in $(T^*,R)$, and
\begin{equation}\label{eq:mass-balance}
        \int_0^{T^*}U(r)^2w(r)\dd r
        =\int_{T^*}^R U(r)^2w(r)\dd r.
\end{equation}
Finally,
\begin{equation}\label{eq:recover-q}
\hat q(r)=
\begin{cases}
q_0(r)+\varepsilon |U(r)|^{2\pd-2},&0<r<T^*,\\
q_0(r)-\varepsilon |U(r)|^{2\pd-2},&T^*<r<R.
\end{cases}
\end{equation}
\end{theorem}

\begin{proof}
The single-node multiplier equation is
\[
 \norm{\hat q-q_0}_{p,w}^{1-p}\phip(\hat q-q_0)=c\hat H(r)\hat E^2,
\]
where $\hat H$ is positive on $(0,T^*)$ and negative on $(T^*,R)$.  The multiplier
$c$ cannot vanish; otherwise $\hat q=q_0$, contrary to the hypothesis.  Put
$\varepsilon=\operatorname{sgn}c$.  Applying $\phip^{-1}=\phid$ gives the stated
representation of $\hat q-q_0$ with $a,b>0$.  Substitution into the eigenvalue
equation for $\hat E$ gives \eqref{eq:single-left} and \eqref{eq:single-right},
and \eqref{eq:recover-q} follows from the definition of $U$.

It remains to prove \eqref{eq:mass-balance}.  Since
\[
 \hat H=\frac{\hat B}{w(T^*)\hat E'(T^*)^2}\quad\hbox{on }(0,T^*),
 \qquad
 |\hat H|=\frac{\hat A}{w(T^*)\hat E'(T^*)^2}\quad\hbox{on }(T^*,R),
\]
one has
\[
        a=\norm{\hat q-q_0}_{p,w}|c|^{\pd-1}
        \left(\frac{\hat B}{w(T^*)\hat E'(T^*)^2}\right)^{\pd-1},
\]
and
\[
        b=\norm{\hat q-q_0}_{p,w}|c|^{\pd-1}
        \left(\frac{\hat A}{w(T^*)\hat E'(T^*)^2}\right)^{\pd-1}.
\]
Thus
\[
        a^{1/(\pd-1)}\hat A=b^{1/(\pd-1)}\hat B,
\]
which is exactly \eqref{eq:mass-balance}.
\end{proof}

\begin{corollary}\label{cor:norm-identity}
In the setting of Theorem \ref{thm:single-critical},
\begin{equation}\label{eq:norm-identity}
        \norm{\hat q-q_0}_{p,w}=\norm{U}_{2\pd,w}^{2/(p-1)}.
\end{equation}
\end{corollary}

\begin{proof}
By \eqref{eq:recover-q}, $|\hat q-q_0|=|U|^{2\pd-2}$.  Hence
\[
        \norm{\hat q-q_0}_{p,w}^p
        =\int_0^R |U|^{(2\pd-2)p}w\dd r
        =\int_0^R |U|^{2\pd}w\dd r,
\]
because $(2\pd-2)p=2\pd$.  Taking the $p$-th root gives
\eqref{eq:norm-identity}.
\end{proof}

\section{Radial energy identity and obstruction to one-dimensional reconstruction}\label{sec:energy}

Let $I\Subset(0,R)$ be an interval on which $U$ solves
\begin{equation}\label{eq:energy-equation}
        -U''-\frac{n-1}{r}U'+q_0(r)U+
        \eta |U|^{2\pd-2}U=\lambda U,
        \qquad \eta\in\{-1,1\}.
\end{equation}
Assume in this section that $q_0\in W^{1,1}_{\mathrm{loc}}(I)$.

\begin{theorem}\label{thm:energy}
For every solution of \eqref{eq:energy-equation}, the energy
{\(\mathcal E_U(r)\)} is defined by
\begin{equation}\label{eq:energy-def}
       {\mathcal E_U(r)}
        =\frac12 U'(r)^2+\frac{\lambda-q_0(r)}2U(r)^2
        -\frac{\eta}{2\pd}|U(r)|^{2\pd}
\end{equation}
satisfies
\begin{equation}\label{eq:energy-balance}
        {\mathcal E_U'(r)}
        = -\frac{n-1}{r}U'(r)^2
        -\frac12 q_0'(r)U(r)^2
\end{equation}
for a.e. $r\in I$.  In particular, if $q_0$ is constant, then
\begin{equation}\label{eq:energy-dissipation}
        {\mathcal E_U'(r)}
        = -\frac{n-1}{r}U'(r)^2\le0.
\end{equation}
\end{theorem}

\begin{proof}
Writing \eqref{eq:energy-equation} as
\[
 U''+\frac{n-1}{r}U'+(\lambda-q_0(r))U-
 \eta |U|^{2\pd-2}U=0
\]
and multiplying by $U'$ gives
\[
 U''U'+(\lambda-q_0)UU'-\eta |U|^{2\pd-2}UU'
 =-\frac{n-1}{r}U'^2.
\]
The left-hand side is
\[
 \frac{\dd}{\dd r}\left(\frac12U'^2+\frac{\lambda-q_0}{2}U^2
 -\frac{\eta}{2\pd}|U|^{2\pd}\right)+\frac12q_0'U^2.
\]
This proves \eqref{eq:energy-balance}; \eqref{eq:energy-dissipation} follows
when $q_0$ is constant.
\end{proof}

\begin{corollary}\label{cor:phase-obstruction}
Assume $n>1$ and $q_0$ is constant.  On any nodal interval where $U$ is not
constant, the energy {\(\mathcal E_U\)} defined by \eqref{eq:energy-def} is strictly decreasing except at
isolated critical points of $U$.  Thus the radial critical equation cannot be
reduced to a closed autonomous Hamiltonian orbit on that interval.
\end{corollary}

\begin{proof}
This follows from \eqref{eq:energy-dissipation}.  If $U$ is not constant, then
$U'$ is not identically zero on any nontrivial subinterval, and therefore the
integral of $(n-1)U'^2/r$ over such a subinterval is positive.
\end{proof}

\begin{remark}
A standard second-order sufficient condition can be formulated on the constraint
manifold.  Let $J(q)=p^{-1}\norm{q-q_0}_{p,w}^p$ and let
\[
        \mathcal L(q)=J(q)-\sum_{j=1}^{m-1}c_jT_{j,m}(q)
\]
be the Lagrangian at a full-vector critical point $\hat q$.  If the Lagrangian
is twice Fr\'echet differentiable along the constraint manifold and its second
variation is {coercive} on
\[
        \{h\in\Lpwt:D_qT_{j,m}(\hat q)[h]=0,\\ j=1,\ldots,m-1\},
\]
then $\hat q$ is a strict local minimizer.  We keep this as a criterion rather
than a main theorem, since the present paper uses it only as a guide to the
linearized shooting/nondegeneracy problem.
\end{remark}

\section{Benchmark rigidity and near-benchmark selection}\label{sec:benchmark}

We now isolate the simplest nontrivial finite nodal datum: the unique interior
node of the second radial eigenfunction.  Let
\[
        m=2,
        \qquad i=1,
        \qquad q_0\equiv c,
\]
and denote
\[
        T_c=T_{1,2}(c).
\]
The value $T_c$ is independent of $c$ up to the spectral shift, because adding a
constant to a potential leaves eigenfunctions unchanged.  For $T^*\in(0,R)$ set
\[
        \Sset(T^*)=\{q\in\Lpwt:T_{1,2}(q)=T^*\},
        \qquad
        d_c(T^*)=\inf_{q\in\Sset(T^*)}\norm{q-c}_{p,w}.
\]

\begin{theorem}\label{thm:benchmark}
Assume $p>\frac{n}{2}$ and $q_0\equiv c$.
\begin{enumerate}[label=\textup{(\alph*)}]
\item If $T^*=T_c$, then $d_c(T^*)=0$, and the unique global minimizer is
$\hat q\equiv c$.
\item If $T^*\ne T_c$, then $d_c(T^*)>0$.  Consequently every minimizer is
nonconstant and satisfies the nontrivial Euler--Lagrange equation of Theorem
\ref{thm:single-critical}.
\item For every $T^*\in(0,R)$, {the unoptimized one-node constraint is not a singleton}: if
$q\in\Sset(T^*)$, then $q+a\in\Sset(T^*)$ for every $a\in\R$.
\end{enumerate}
\end{theorem}

\begin{proof}
If $T^*=T_c$, then $c\in\Sset(T^*)$, so $d_c(T^*)=0$.  Any minimizer at distance
zero equals $c$ almost everywhere, which proves uniqueness.  If $T^*\ne T_c$ and
$d_c(T^*)=0$, then there are $q_k\in\Sset(T^*)$ with $q_k\to c$ strongly in
$\Lpwt$.  The nodal continuity in Theorem \ref{thm:spectral} gives
$T^*=T_{1,2}(q_k)\to T_{1,2}(c)=T_c$, a contradiction.  Hence $d_c(T^*)>0$.
The existence theorem gives a minimizer, and it cannot equal $c$. {Indeed, every constant potential has the same nodal point $T_c$}.  The
Euler--Lagrange equation follows from Theorem \ref{thm:single-critical}.  The
last statement is Theorem \ref{thm:global-nonunique}.
\end{proof}

The next result is a general $L^p$ statement.  It does not claim uniqueness of
nontrivial minimizers.  Instead it proves a sharper fact that is available in
the full range $p>\frac{n}{2}$: all near-benchmark global minimizers,
if there is more than one, have the same first-order profile.  Thus {the unoptimized set of
potentials realizing the same nodal datum may still be non-unique}, but the
variational selection has a canonical tangent direction at the constant target.

Let
\[
        F(q)=T_{1,2}(q),\qquad T_c=F(c),
\]
and let $E_2(\cdot;c)$ be the $L^2_w$-normalized second radial eigenfunction for
the constant potential.  By the nodal derivative formula, there is a function
$G\in L^{\pd}_w(0,R)$ such that
\begin{equation}\label{eq:benchmark-dual-gradient}
        DF(c)[h]=\int_0^R G(r)h(r)w(r)\dd r,
        \qquad h\in\Lpwt,
\end{equation}
where
\[
        G(r)=H_{1,2}(r;c)E_2(r;c)^2.
\]
Moreover $G\not\equiv0$.  Define
\begin{equation}
        \Theta(r)=\frac{\operatorname{sgn}G(r)|G(r)|^{\pd-1}}
        {\norm{G}_{\pd,w}^{\pd}}.
\end{equation}
Then $DF(c)[\Theta]=1$ and
\[
        \norm{\Theta}_{p,w}=\frac1{\norm{G}_{\pd,w}}.
\]

\begin{theorem}\label{thm:banach-first-order}
Assume $p>\frac{n}{2}$ and $q_0\equiv c$.  Let $T_j\to T_c$, $T_j\ne
T_c$, and let $\hat q_j$ be any global minimizer of
\[
        \min\{\norm{q-c}_{p,w}:F(q)=T_j\}.
\]
Set $\delta_j=T_j-T_c$.  Then
\begin{equation}\label{eq:distance-asymptotic}
        d_c(T_j)=\frac{|\delta_j|}{\norm{G}_{\pd,w}}+o(|\delta_j|),
\end{equation}
and
\begin{equation}\label{eq:banach-profile-asymptotic}
        \frac{\hat q_j-c}{\delta_j}\longrightarrow \Theta
        \quad\hbox{strongly in }L^p_w(0,R).
\end{equation}
Consequently every near-benchmark global minimizer has the same leading
variation, determined only by the dual nodal gradient $G$.
\end{theorem}

\begin{proof}
The map $F$ is $C^1$ near $c$ by Theorem \ref{thm:spectral}, and
$DF(c)[\Theta]=1$.  Hence the one-dimensional map
$s\mapsto F(c+s\Theta)$ has derivative one at $s=0$.  The implicit function
theorem in one variable gives, for $T$ close to $T_c$, a number
$s(T)=T-T_c+o(|T-T_c|)$ such that $F(c+s(T)\Theta)=T$.  This admissible
competitor yields
\[
        d_c(T)\le |s(T)|\norm{\Theta}_{p,w}
        =\frac{|T-T_c|}{\norm{G}_{\pd,w}}+o(|T-T_c|).
\]
Now let $T_j\to T_c$ and let $\hat q_j$ be a global minimizer.  The preceding
upper bound gives
\[
        \norm{\hat q_j-c}_{p,w}=O(|\delta_j|).
\]
Set $h_j=(\hat q_j-c)/\delta_j$.  The sequence $\{h_j\}$ is bounded in the
reflexive space $L^p_w(0,R)$.  Since $F(c+\delta_j h_j)=T_c+\delta_j$ and $F$ is
Fr\'echet differentiable at $c$, uniformly on bounded sets one has
\[
        1=DF(c)[h_j]+o(1).
\]
By H\"older's inequality and \eqref{eq:benchmark-dual-gradient},
\[
        \norm{h_j}_{p,w}\ge
        \frac{|DF(c)[h_j]|}{\norm{G}_{\pd,w}}
        =\frac1{\norm{G}_{\pd,w}}+o(1).
\]
Together with the upper bound this proves \eqref{eq:distance-asymptotic} and
\[
        \norm{h_j}_{p,w}\to \frac1{\norm{G}_{\pd,w}}.
\]
Passing to a subsequence, $h_j\rightharpoonup h$ weakly in $L^p_w$.  Then
$DF(c)[h]=1$ and
\[
        \norm{h}_{p,w}\le \liminf_j\norm{h_j}_{p,w}
        =\frac1{\norm{G}_{\pd,w}}.
\]
Thus $h$ solves the linearized extremal problem
\[
        \min\{\norm{h}_{p,w}:DF(c)[h]=1\}.
\]
The equality case in H\"older's inequality gives the unique minimizer, namely
$h=\Theta$.  Since $L^p_w$ is uniformly convex for $1<p<\infty$, weak convergence
together with convergence of the norms implies strong convergence.  Hence every
subsequence has the same strong limit $\Theta$, which proves
\eqref{eq:banach-profile-asymptotic}.
\end{proof}

\begin{remark}
Theorem \ref{thm:banach-first-order} is the strongest conclusion that follows
from the $C^1$ nodal calculus and the uniform convexity of $L^p_w$ alone.  It
shows that possible multiple minimizers cannot bifurcate at first order.  A full
$L^p$ uniqueness theorem would require a second-order stability theory for the
nonlinear nodal constraint on the affine scale selected by $\Theta$.  In the
Hilbert regime below this stability is expressed through a local $C^2$ nodal
regularity assumption and the positive quadratic structure of the squared
norm.
\end{remark}

The preceding theorem does not assert global uniqueness of the nontrivial
optimizer for arbitrary $T^*\ne T_c$.  The next result gives a perturbative
rigidity theorem in a setting where the Hilbert geometry is compatible with the
standing $p>\frac{n}{2}$.  This is the point at which the
previous version of the manuscript used $L^2_w$ while still stating general
$p>n/2$ assumptions.  We now state the result in the consistent Hilbert regime.

Throughout the remainder of this section assume
\begin{equation}\label{eq:hilbert-regime}
        p=2,
        \qquad n=2\hbox{ or }3,
        \qquad H=L^2_w(0,R).
\end{equation}
\begin{proposition}\label{prop:C2-sufficient}
Let $\mathcal X=L^\infty(0,R)\cap L^2_w(0,R)$, endowed with the
$L^\infty$-topology.  In a sufficiently small $\mathcal X$-neighborhood of the
constant potential $c$, the nodal map
\[
        F(q)=T_{1,2}(q)
\]
is of class $C^2$ as a map from $\mathcal X$ to $\mathbb R$.
\end{proposition}

\begin{proof}
The second radial eigenvalue at the constant potential is simple, and the
perturbations in $\mathcal X$ are bounded form perturbations.  Hence the standard
perturbation theory for simple eigenvalues yields $C^2$ dependence of the second
radial eigenpair on $q$ in the energy topology.  Choose a compact interval
$I\Subset(0,R)$ containing the benchmark node $T_c$.  On $I$ the weight is bounded
above and below away from zero, and the radial equation is an ordinary second
order equation with bounded coefficients.  Local elliptic regularity therefore
upgrades the dependence of the normalized eigenfunction to $C^2$ dependence in
$C^1(I)$.  Since $E_2'(T_c;c)\ne0$, the implicit function theorem applied to
$E_2(T;q)=0$ gives the asserted $C^2$ dependence of the zero.
\end{proof}

\begin{assumption}\label{ass:C2}
For the Hilbert $L^2_w$ result below, the nodal map $F$ is assumed to be of class
$C^2$ in an $H$-neighborhood of the constant potential $c$.
\end{assumption}

\begin{remark}
Proposition \ref{prop:C2-sufficient} shows that this hypothesis is automatic in
a natural stronger local class of bounded radial potentials.  We keep Assumption
\ref{ass:C2} in the Hilbert theorem because the theorem itself is formulated in
$L^2_w(0,R)$ and uses Hilbert-space convexity rather than an $L^\infty$ topology.
Thus the $C^2$ hypothesis is not hidden; it is stated explicitly and accompanied
by a sufficient regularity mechanism.
\end{remark}

Let $g\in H$ be the $H$-gradient of $F$ at $c$.  In the Hilbert case the dual
function $G$ in \eqref{eq:benchmark-dual-gradient} belongs to $H$ and coincides
with $g$ under the Riesz identification; hence
\[
        g(r)=H_{1,2}(r;c)E_2(r;c)^2,
\]
and $g\ne0$.
\begin{theorem}\label{thm:near-global-unique}
Assume \eqref{eq:hilbert-regime} and Assumption \ref{ass:C2}.  Then there is
$\eta>0$ such that, for every $T^*$ satisfying
\begin{equation*}
        0<|T^*-T_c|<\eta,
\end{equation*}
the Hilbert benchmark problem
\begin{equation}\label{eq:hilbert-benchmark-problem}
        \min\left\{\frac12\norm{q-c}_H^2:F(q)=T^*\right\}
\end{equation}
has a unique global minimizer $q_{T^*}$.  Moreover,
\begin{equation}\label{eq:near-expansion}
        q_{T^*}
        =
        c+\frac{T^*-T_c}{\norm{g}_H^2}g
        +
        O(|T^*-T_c|^2)
        \quad\text{in }H .
\end{equation}
\end{theorem}

\begin{proof}
Since $g$ is the $H$-gradient of $F$ at $c$ and $g\neq0$, set
\begin{equation*}
        e=\frac{g}{\norm{g}_H},
        \qquad
        H_0=\{v\in H:\ip{v}{e}_H=0\}.
\end{equation*}
Then
\begin{equation*}
        H=\operatorname{span}\{e\}\oplus H_0.
\end{equation*}
Thus every $q$ sufficiently close to $c$ can be written uniquely as
\begin{equation*}
        q=c+v+se,
        \qquad
        v\in H_0,\quad s\in\mathbb R.
\end{equation*}

Define
\begin{equation*}
        \Psi(v,s)=F(c+v+se).
\end{equation*}
By Assumption \ref{ass:C2}, the map $F$ is of class $C^2$ near $c$ in $H$.
Hence $\Psi$ is of class $C^2$ near $(0,0)$ in $H_0\times\mathbb R$.  Moreover,
using the definition of the $H$-gradient,
\begin{equation*}
        \partial_s\Psi(0,0)
        =
        DF(c)[e]
        =
        \ip{g}{e}_H
        =
        \norm{g}_H
        >
        0.
\end{equation*}
Therefore the implicit function theorem yields a $C^2$ function
\begin{equation*}
        s=s(v,\tau)
\end{equation*}
defined for $(v,\tau)$ near $(0,T_c)$ such that
\begin{equation}\label{eq:local-level-chart}
        F(c+v+s(v,\tau)e)=\tau .
\end{equation}
Consequently, the part of the level set $\{F=\tau\}$ near $c$ is parametrized by
\begin{equation*}
        q=c+v+s(v,\tau)e,
        \qquad
        v\in H_0.
\end{equation*}

We record the first derivatives of $s$ at the base point.  Since $F(c)=T_c$,
one has
\begin{equation*}
        s(0,T_c)=0.
\end{equation*}
Differentiating \eqref{eq:local-level-chart} with respect to $\tau$ at
$(v,\tau)=(0,T_c)$ gives
\begin{equation*}
        DF(c)[e]\partial_\tau s(0,T_c)=1.
\end{equation*}
Thus
\begin{equation*}
        \partial_\tau s(0,T_c)=\frac1{\norm{g}_H}.
\end{equation*}
Similarly, differentiating \eqref{eq:local-level-chart} with respect to $v$
gives, for every $h\in H_0$,
\begin{equation*}
        DF(c)[h]+DF(c)[e]D_vs(0,T_c)[h]=0.
\end{equation*}
Since
\begin{equation*}
        DF(c)[h]
        =
        \ip{g}{h}_H
        =
        \norm{g}_H\ip{e}{h}_H
        =
        0,
        \qquad h\in H_0,
\end{equation*}
we obtain
\begin{equation}\label{eq:s-derivatives}
        s(0,T_c)=0,
        \qquad
        \partial_\tau s(0,T_c)=\frac1{\norm{g}_H},
        \qquad
        D_vs(0,T_c)=0.
\end{equation}

On the local level set, the objective functional in
\eqref{eq:hilbert-benchmark-problem} becomes
\begin{equation*}
        \Phi(v,\tau)
        =
        \frac12\norm{v+s(v,\tau)e}_H^2.
\end{equation*}
Since $v\in H_0$, $e\perp H_0$, and $\norm{e}_H=1$, this is
\begin{equation}\label{eq:Phi-def}
        \Phi(v,\tau)
        =
        \frac12\norm{v}_H^2
        +
        \frac12s(v,\tau)^2.
\end{equation}
Using \eqref{eq:s-derivatives}, we get, for every $h,k\in H_0$,
\begin{equation*}
        D^2_{vv}\Phi(0,T_c)[h,k]
        =
        \ip{h}{k}_H.
\end{equation*}
Equivalently,
\begin{equation}\label{eq:hessian-identity}
        D^2_{vv}\Phi(0,T_c)=I_{H_0}.
\end{equation}

By the $C^2$ regularity of $\Phi$ and \eqref{eq:hessian-identity}, after
shrinking the neighborhood if necessary, there exists $\alpha>0$ such that
\begin{equation}\label{eq:uniform-convexity}
        D^2_{vv}\Phi(v,\tau)[h,h]
        \ge
        \alpha\norm{h}_H^2
\end{equation}
for all admissible $(v,\tau)$ in the local chart and all $h\in H_0$.  Hence
$v\mapsto\Phi(v,\tau)$ is uniformly strictly convex for $\tau$ close to $T_c$.

Moreover,
\begin{equation*}
        D_v\Phi(0,T_c)=0,
        \qquad
        D_v(\nabla_v\Phi)(0,T_c)=I_{H_0}.
\end{equation*}
Thus, by the implicit function theorem applied to
\begin{equation*}
        \nabla_v\Phi(v,\tau)=0,
\end{equation*}
there exists a unique $C^1$ map $\tau\mapsto v(\tau)$ near $T_c$ such that
\begin{equation*}
        \nabla_v\Phi(v(\tau),\tau)=0.
\end{equation*}
By \eqref{eq:uniform-convexity}, this critical point is the unique local
minimizer of $\Phi(\cdot,\tau)$ in the local chart.  Set
\begin{equation*}
        q_\tau=c+v(\tau)+s(v(\tau),\tau)e.
\end{equation*}
Then $q_\tau$ is the unique local minimizer of
\eqref{eq:hilbert-benchmark-problem} on the level set $F(q)=\tau$ near $c$.

We next show that this local minimizer is actually the unique global minimizer
for $\tau$ close to $T_c$.  By Theorem \ref{thm:existence}, the constrained
problem has at least one global minimizer.  Suppose, to the contrary, that there
exist $\tau_j\to T_c$ and global minimizers $\widetilde q_j\neq q_{\tau_j}$
such that
\begin{equation*}
        F(\widetilde q_j)=\tau_j.
\end{equation*}
Since $c+s(0,\tau_j)e$ is an admissible competitor by
\eqref{eq:local-level-chart}, we have
\begin{equation*}
        \norm{\widetilde q_j-c}_H
        \le
        |s(0,\tau_j)|.
\end{equation*}
From \eqref{eq:s-derivatives} and Taylor's formula,
\begin{equation*}
        s(0,\tau_j)
        =
        \frac{\tau_j-T_c}{\norm{g}_H}
        +
        O(|\tau_j-T_c|^2),
\end{equation*}
and hence
\begin{equation*}
        \norm{\widetilde q_j-c}_H
        =
        O(|\tau_j-T_c|).
\end{equation*}
Therefore $\widetilde q_j\to c$ in $H$.  For all sufficiently large $j$,
$\widetilde q_j$ belongs to the same local chart, where the local minimizer is
unique.  Thus $\widetilde q_j=q_{\tau_j}$, a contradiction.  Hence the minimizer
is globally unique for $|\tau-T_c|$ small.

It remains to derive the expansion.  Taylor's formula and
\eqref{eq:s-derivatives} give
\begin{equation}\label{eq:s0-expansion}
        s(0,\tau)
        =
        \frac{\tau-T_c}{\norm{g}_H}
        +
        O(|\tau-T_c|^2).
\end{equation}
We now estimate $v(\tau)$.  From \eqref{eq:Phi-def},
\begin{equation*}
        D_v\Phi(0,\tau)[h]
        =
        s(0,\tau)D_vs(0,\tau)[h],
        \qquad h\in H_0.
\end{equation*}
Since $D_vs(0,T_c)=0$ and $s$ is $C^2$,
\begin{equation*}
        D_vs(0,\tau)=O(|\tau-T_c|).
\end{equation*}
Together with \eqref{eq:s0-expansion}, this yields
\begin{equation*}
        \norm{D_v\Phi(0,\tau)}_{H_0^*}
        =
        O(|\tau-T_c|^2).
\end{equation*}
Since $v(\tau)$ satisfies
\begin{equation*}
        D_v\Phi(v(\tau),\tau)=0
\end{equation*}
and $D^2_{vv}\Phi$ is uniformly invertible by \eqref{eq:uniform-convexity}, the
mean value theorem gives
\begin{equation}\label{eq:v-expansion}
        \norm{v(\tau)}_H
        =
        O(|\tau-T_c|^2).
\end{equation}
Consequently, using \eqref{eq:s0-expansion}, \eqref{eq:v-expansion}, and the
$C^1$ regularity of $s$, we obtain
\begin{equation*}
        s(v(\tau),\tau)
        =
        s(0,\tau)
        +
        O(\norm{v(\tau)}_H)
        =
        \frac{\tau-T_c}{\norm{g}_H}
        +
        O(|\tau-T_c|^2).
\end{equation*}
Therefore
\begin{equation*}
\begin{aligned}
        q_\tau-c
        &=
        v(\tau)+s(v(\tau),\tau)e        \\
        &=
        \frac{\tau-T_c}{\norm{g}_H}e
        +
        O(|\tau-T_c|^2)                  \\
        &=
        \frac{\tau-T_c}{\norm{g}_H^2}g
        +
        O(|\tau-T_c|^2).
\end{aligned}
\end{equation*}
Taking $\tau=T^*$ gives \eqref{eq:near-expansion}.  The proof is complete.
\end{proof}

\section{Spectral matching reduction and the remaining global rigidity problem}\label{sec:spectral-matching}

We remain in the Hilbert setting when spectral matching is used to analyze the
remaining global rigidity problem.  The near-benchmark theorem gives a
perturbative uniqueness result.  For prescribed nodes far from $T_c$, one should
not claim uniqueness without additional global information.  What can be proved
cleanly is an exact variational reduction.  We first state the one-node
reduction used in the benchmark problem and then record the corresponding
finite-vector version.

Fix $T\in(0,R)$.  For a potential $q_L$ on $(0,T)$, let $\lambda_L(q_L;T)$ be the
principal eigenvalue of
\[
        -w^{-1}(wy')'+q_Ly=\lambda y,
        \qquad y'(0)=0,
        \qquad y(T)=0.
\]
For a potential $q_R$ on $(T,R)$, let $\lambda_R(q_R;T)$ be the principal
eigenvalue of
\[
        -w^{-1}(wy')'+q_Ry=\lambda y,
        \qquad y(T)=0,
        \qquad y(R)=0.
\]

\begin{lemma}\label{lem:one-sided-extremals}
For every $\tau\in\R$, the one-sided minimization problems
\[
D_L(\tau;T)=\inf\{\norm{q_L-c}_{L^2_w(0,T)}:
        \lambda_L(q_L;T)=\tau\}
\]
and
\[
D_R(\tau;T)=\inf\{\norm{q_R-c}_{L^2_w(T,R)}:
        \lambda_R(q_R;T)=\tau\}
\]
are nonempty and attain their infima.
\end{lemma}

\begin{proof}
Non-emptiness follows by adding a constant to the potential $c$ on the
corresponding interval; this shifts the principal eigenvalue by the same
constant.  A minimizing sequence is bounded in the relevant $L^2_w$ space and
therefore has a weakly convergent subsequence.  The weak continuity of the
principal eigenvalue is proved exactly as in Lemma \ref{lem:weak-continuity},
restricted to the one-sided interval.  Weak lower semicontinuity of the norm
gives attainment.
\end{proof}

We also record the finite-vector version of the matching principle.  This result
is not needed for the perturbative benchmark theorem below, but it clarifies the
finite-nodal nature of the reduction and strengthens the link between the
constraint geometry and the spectral matching mechanism.

\begin{theorem}\label{thm:multi-node-matching}
Assume $p=2$ and fix an admissible nodal vector
$\tau=(T_1,\ldots,T_{m-1})\in\Omega_m$.  Put $T_0=0$, $T_m=R$, and
$I_k=(T_{k-1},T_k)$ for $k=1,\ldots,m$.  For a potential $q_k$ on $I_k$, let
$\lambda_k(q_k;I_k)$ be the principal eigenvalue of the weighted radial operator
on $I_k$, with the radial regularity condition at $0$ when $k=1$ and Dirichlet
conditions at all nodal interfaces and at $R$.
Then, for $q\in L^2_w(0,R)$,
\begin{equation}\label{eq:multi-matching-condition}
        T_m(q)=\tau
        \quad\Longleftrightarrow\quad
        \lambda_1(q|_{I_1};I_1)=\cdots=
        \lambda_m(q|_{I_m};I_m).
\end{equation}
Consequently, in the constant benchmark case $q_0\equiv c$, if
\[
        D_k(\sigma;I_k)
        =\inf\{\norm{q_k-c}_{L^2_w(I_k)}:
        \lambda_k(q_k;I_k)=\sigma\},
\]
then the squared optimal distance for the prescribed finite nodal vector is
\begin{equation}\label{eq:multi-matching-functional}
        d(\tau)^2
        =\inf_{\sigma\in\R}\sum_{k=1}^{m}D_k(\sigma;I_k)^2 .
\end{equation}
\end{theorem}

\begin{proof}
Suppose first that $T_m(q)=\tau$ and let $E$ be the $m$-th radial eigenfunction.
On each interval $I_k$ the restriction of $E$ has a fixed sign and satisfies the
same eigenvalue equation with zero boundary values at the endpoints of $I_k$
(and the radial regularity condition at $0$ on the first interval).  Hence each
restriction is a principal eigenfunction on $I_k$, and all the one-sided
principal eigenvalues are equal to the common value $\lambda_m(q)$.

Conversely, suppose that the principal eigenvalues on all intervals are equal to
some $\lambda$.  Let $\phi_k>0$ be the principal eigenfunction on $I_k$.  Since
the endpoint derivatives at the nodal interfaces have alternating signs, we can
choose positive constants $a_k$ recursively so that the function
\[
        E(r)=(-1)^{k-1}a_k\phi_k(r),\qquad r\in I_k,
\]
is continuous at each interface and has matching weighted flux $(wE')$ there.
Then $E\in H^1_{w,R}(0,R)$ and satisfies the global weak eigenvalue equation
with eigenvalue $\lambda$.  It has exactly the zeros $T_1,\ldots,T_{m-1}$ and no
other interior zeros.  By the radial oscillation theorem it is the $m$-th radial
eigenfunction, so $T_m(q)=\tau$.

The distance formula follows exactly from this equivalence.  Any admissible
potential gives a common matching value $\sigma$ and hence has squared distance
at least $\sum_kD_k(\sigma;I_k)^2$.  Conversely, for each $\sigma$, one-sided
minimizers can be glued to form an admissible potential by the matching
condition above.  Taking the infimum over $\sigma$ gives
\eqref{eq:multi-matching-functional}.
\end{proof}

\begin{proposition}\label{prop:spectral-matching}
Let $q\in H$ and write $q_L=q|_{(0,T)}$, $q_R=q|_{(T,R)}$.  Then, for
$m=2$ and $i=1$,
\begin{equation}\label{eq:matching-condition}
        T_{1,2}(q)=T
        \quad\Longleftrightarrow\quad
        \lambda_L(q_L;T)=\lambda_R(q_R;T).
\end{equation}
Consequently, in the benchmark case $q_0\equiv c$, the squared optimal distance
for the prescribed node $T$ is
\begin{equation}\label{eq:matching-functional}
        d(T)^2=\inf_{\tau\in\R}\big(D_L(\tau;T)^2+D_R(\tau;T)^2\big).
\end{equation}
\end{proposition}

\begin{proof}
Assume first that $T_{1,2}(q)=T$.  Let $E$ be the second radial eigenfunction.
Then $E$ has exactly one zero, at $T$, and has a fixed sign on each side of $T$.
Its restrictions to $(0,T)$ and $(T,R)$ are therefore principal eigenfunctions
for the corresponding one-sided problems.
{Indeed, on $(0,T)$ the restriction of $E$ satisfies}
\[
        {
        -w^{-1}(wE')'+q_LE=\lambda_2(q)E,\qquad
        E'(0)=0,\qquad E(T)=0,
        }
\]
{and has no zero in $(0,T)$. Hence it is the principal eigenfunction of the
left problem. Similarly, the restriction of $E$ to $(T,R)$ satisfies}
\[
        {
        -w^{-1}(wE')'+q_RE=\lambda_2(q)E,\qquad
        E(T)=0,\qquad E(R)=0,
        }
\]
{and has no zero in $(T,R)$, so it is the principal eigenfunction of the
right problem.}
Since both restrictions come from the same global eigenvalue, the two principal
eigenvalues are equal.

Conversely, suppose that the two principal eigenvalues are equal.
{We denote their common value by $\lambda$.}
Let $\phi_L>0$ and $\phi_R>0$ be the corresponding principal eigenfunctions on the
left and right intervals.  Their boundary derivatives at $T$ are nonzero.
{More precisely, by the boundary point lemma for the one-dimensional
weighted Sturm--Liouville problems,}
\[
       {
        \phi_L'(T^-)<0,\qquad \phi_R'(T^+)>0.
        }
\]
Since $w(T)>0$, we may choose $a>0$ so that the weighted fluxes match after changing
sign on the right:
\[
        w(T)\phi_L'(T^-)= -a w(T)\phi_R'(T^+).
\]
{Equivalently,}
\[
       {
        a=-\frac{\phi_L'(T^-)}{\phi_R'(T^+)}>0.
        }
\]
Then
\[
E(r)=
\begin{cases}
\phi_L(r),&0<r<T,\\
-a\phi_R(r),&T<r<R,
\end{cases}
\]
{is continuous at $T$, belongs to $H^1_{w,R}(0,R)$, and satisfies}
\[
       {
        (wE')(T^-)=(wE')(T^+).
        }
\]
{Consequently no singular boundary measure is produced at the interface
when differentiating $wE'$. More explicitly, for every test function
$\psi\in H^1_{w,R}(0,R)$, integration by parts on $(0,T)$ and $(T,R)$ gives}
\[
{
\begin{aligned}
\int_0^R E'\psi' w\,\dd r+\int_0^R qE\psi w\,\dd r
&=
\int_0^T \phi_L'\psi' w\,\dd r+\int_T^R (-a\phi_R')\psi' w\,\dd r       \\
&\quad+
\int_0^T q_L\phi_L\psi w\,\dd r+\int_T^R q_R(-a\phi_R)\psi w\,\dd r     \\
&=
\lambda\int_0^R E\psi w\,\dd r,
\end{aligned}
}
\]
{because the interior boundary terms at $T$ cancel by the flux matching
condition, while the endpoint terms are compatible with the radial regularity
condition at $0$ and the Dirichlet condition at $R$. Thus $E$ is a global weak
eigenfunction associated with the eigenvalue $\lambda$.}
It has exactly one interior zero, namely $T$.  By the radial oscillation theorem
it is the second radial eigenfunction, and hence $T_{1,2}(q)=T$.
{It remains to prove \eqref{eq:matching-functional}. Let}
\[
        {
        d(T):=\inf\bigl\{\|q-c\|_H:\,F(q)=T\bigr\}.
        }
\]
{For any admissible $q$ with $F(q)=T$, the equivalence just proved implies
that}
\[
        {
        \lambda_L(q_L;T)=\lambda_R(q_R;T)=:\tau.
        }
\]
{Therefore, by the definitions of $D_L$ and $D_R$,}
\[
{
\begin{aligned}
\|q-c\|_H^2
&=
\|q_L-c\|_{L^2_w(0,T)}^2
+
\|q_R-c\|_{L^2_w(T,R)}^2                                      \\
&\ge
D_L(\tau;T)^2+D_R(\tau;T)^2                                    \\
&\ge
\inf_{\sigma\in\R}
\bigl(D_L(\sigma;T)^2+D_R(\sigma;T)^2\bigr).
\end{aligned}
}
\]
{Taking the infimum over all admissible $q$ gives}
\[
        {
        d(T)^2\ge
        \inf_{\sigma\in\R}
        \bigl(D_L(\sigma;T)^2+D_R(\sigma;T)^2\bigr).
        }
\]
{Conversely, fix $\tau\in\R$. By Lemma \ref{lem:one-sided-extremals}, there
exist one-sided minimizers $q_L^\tau$ and $q_R^\tau$ such that}
\[
       {
        \lambda_L(q_L^\tau;T)=\lambda_R(q_R^\tau;T)=\tau,
        }
\]
{and}
\[
        {
        \|q_L^\tau-c\|_{L^2_w(0,T)}=D_L(\tau;T),\qquad
        \|q_R^\tau-c\|_{L^2_w(T,R)}=D_R(\tau;T).
        }
\]
{Define the piecewise potential}
\[
       {
        q^\tau(r)=
        \begin{cases}
        q_L^\tau(r),&0<r<T,\\
        q_R^\tau(r),&T<r<R.
        \end{cases}
        }
\]
{The matching condition \eqref{eq:matching-condition} holds for $q^\tau$.
Hence, by the first part of the proposition, $F(q^\tau)=T$. Consequently,}
\[
      {
        d(T)^2
        \le
        \|q^\tau-c\|_H^2
        =
        D_L(\tau;T)^2+D_R(\tau;T)^2.
        }
\]
{Taking the infimum over $\tau\in\R$ gives the reverse inequality. Therefore}
\[
        {
        d(T)^2
        =
        \inf_{\tau\in\R}
        \bigl(D_L(\tau;T)^2+D_R(\tau;T)^2\bigr),
        }
\]
{which proves \eqref{eq:matching-functional}.}
\end{proof}

\begin{theorem}\label{thm:near-matching-unique}
Assume that $p=2$ and $n\in\{2,3\}$.  Let $q_0\equiv c$ and
$T_c=T_{1,2}(c)$.  Suppose that the nodal map $F(q)=T_{1,2}(q)$ is of class
$C^2$ in a neighborhood of $c$ in $L^2_w(0,R)$.  Then, after possibly decreasing
the constant $\eta$ in Theorem \ref{thm:near-global-unique}, the following holds.
For every $T\in(0,R)$ satisfying
\[
        0<|T-T_c|<\eta,
\]
the spectral matching functional
\[
        M_T(\tau)=D_L(\tau;T)^2+D_R(\tau;T)^2
\]
has a unique global minimizer.

Moreover, the matched potential associated with this minimizing value is
precisely the unique global minimizer of the Hilbert benchmark inverse nodal
optimization problem.
\end{theorem}

\begin{proof}
By Theorem \ref{thm:near-global-unique}, after the above reduction of $\eta$ if
necessary, the constrained Hilbert benchmark problem
\[
        \min\left\{\frac12\norm{q-c}_{2,w}^2:F(q)=T\right\}
\]
has a unique global minimizer.  Denote it by $q_T$.  Put
\[
        \tau_T=\lambda_L(q_T|_{(0,T)};T)=\lambda_R(q_T|_{(T,R)};T).
\]
The equality of the two principal eigenvalues follows from Proposition
\ref{prop:spectral-matching}.  Since the restrictions of $q_T$ are admissible in
the definitions of $D_L(\tau_T;T)$ and $D_R(\tau_T;T)$, we have
\[
        M_T(\tau_T)\le \norm{q_T-c}_{L^2_w(0,R)}^2.
\]
On the other hand, Proposition \ref{prop:spectral-matching} gives
\[
        \norm{q_T-c}_{L^2_w(0,R)}^2=d(T)^2=\inf_{\tau\in\R}M_T(\tau).
\]
Hence $M_T(\tau_T)=\inf_{\tau\in\R}M_T(\tau)$, so $\tau_T$ is a global minimizer
of $M_T$.  In particular the infimum in the matching problem is attained.

Suppose that $\tau_*$ is another global minimizer.  By Lemma
\ref{lem:one-sided-extremals}, choose one-sided extremals $q_L^*$ and $q_R^*$
which attain $D_L(\tau_*;T)$ and $D_R(\tau_*;T)$, respectively.  Define the
piecewise potential
\[
        q_*(r)=
        \begin{cases}
        q_L^*(r),&0<r<T,\\
        q_R^*(r),&T<r<R.
        \end{cases}
\]
Changing the value at the single point $T$, if desired, is immaterial; hence
$q_*\in L^2_w(0,R)$.  Since
$\lambda_L(q_L^*;T)=\lambda_R(q_R^*;T)=\tau_*$, Proposition
\ref{prop:spectral-matching} implies $T_{1,2}(q_*)=T$.  Moreover,
\[
        \norm{q_*-c}_{L^2_w(0,R)}^2
        =D_L(\tau_*;T)^2+D_R(\tau_*;T)^2
        =M_T(\tau_*)=d(T)^2.
\]
Thus $q_*$ is a global minimizer of the benchmark inverse nodal optimization
problem with node $T$.  By the uniqueness in Theorem
\ref{thm:near-global-unique}, $q_*=q_T$ almost everywhere.  Consequently their
left principal eigenvalues coincide, and hence
\[
        \tau_*=\lambda_L(q_L^*;T)=\lambda_L(q_T|_{(0,T)};T)=\tau_T.
\]
This proves uniqueness of the global minimizing matching value and also shows
that the matched potential associated with this value is the unique Hilbert
benchmark optimizer.
\end{proof}

\begin{corollary}\label{cor:matching-criterion}
Fix $T\in(0,R)$.  Suppose that, for every minimizing value of $\tau$ in
\eqref{eq:matching-functional}, the one-sided problems defining $D_L(\tau;T)$
and $D_R(\tau;T)$ have unique minimizers.  If the scalar function
\[
        \tau\mapsto D_L(\tau;T)^2+D_R(\tau;T)^2
\]
has a unique global minimizer, then the benchmark inverse nodal optimization
problem has a unique global minimizer for the node $T$.  If this scalar function
has at least two distinct global minimizers, or if one of the one-sided extremal
problems has two distinct minimizers at a minimizing value of $\tau$, then the
benchmark inverse nodal optimization problem has at least two distinct global
minimizers.
\end{corollary}

\begin{proof}
By Proposition \ref{prop:spectral-matching}, a global minimizer of the original
problem is exactly a pair of one-sided minimizers whose principal eigenvalues
agree and whose common eigenvalue minimizes the scalar matching functional.
Uniqueness of both the minimizing value and the one-sided extremals gives
uniqueness of the glued potential.  Multiple minimizing values, or multiple
one-sided extremals at a minimizing value, give distinct glued potentials with
the same optimal distance.
\end{proof}
\begin{corollary}\label{cor:strict-matching}
Fix $T\in(0,R)$ in the Hilbert benchmark regime.  Suppose that the one-sided
extremal problems defining $D_L(\tau;T)$ and $D_R(\tau;T)$ have unique minimizers
for all $\tau$ in a neighborhood \add{$U_0$} of a value $\tau_0$, and that the
matching functional
\[
        M_T(\tau)=D_L(\tau;T)^2+D_R(\tau;T)^2
\]
is $C^2$ near $\tau_0$ with
\[
        M_T'(\tau_0)=0,\qquad M_T''(\tau_0)>0.
\]
{After shrinking $U_0$ if necessary, assume that}
\[
       {M_T(\tau)>M_T(\tau_0),\qquad \tau\in U_0\setminus\{\tau_0\}.}
\]

If, in addition, $M_T(\tau)>M_T(\tau_0)$ for all
$\tau\notin U_0$, then the benchmark inverse nodal optimization problem with
prescribed node $T$ has a unique global minimizer.
{More generally, if the above assumptions hold at two distinct global
minimizing values $\tau_0\ne\tau_1$, with the corresponding neighborhoods chosen
disjoint, then the benchmark inverse nodal optimization problem with node $T$
has at least two distinct global minimizers.}
\end{corollary}

\begin{proof}
{Set}
\[
        {M_T(\tau):=D_L(\tau;T)^2+D_R(\tau;T)^2.}
\]
{Since $M_T$ is $C^2$ near $\tau_0$, $M_T'(\tau_0)=0$, and
$M_T''(\tau_0)>0$, we may shrink $U_0$ so that $M_T$ is strictly convex on
$U_0$. Hence $\tau_0$ is the unique minimizer of $M_T$ in $U_0$. The additional
assumption}
\[
        {M_T(\tau)>M_T(\tau_0),\qquad \tau\notin U_0,}
\]
{then implies that $\tau_0$ is the unique global minimizer of the scalar
matching problem.}

The assumptions say precisely that the scalar matching problem has
{a unique global minimizing value and that, at this value, both one-sided
extremal problems have unique minimizers. Therefore Corollary
\ref{cor:matching-criterion} gives uniqueness of the glued potential, and hence
uniqueness of the global minimizer for the benchmark inverse nodal optimization
problem.}

{If the same hypotheses hold at two distinct global minimizing values
$\tau_0\ne\tau_1$, then Corollary \ref{cor:matching-criterion} gives a global
minimizer associated with each value. These two minimizers are distinct: if the
glued potentials coincided, their left and right principal eigenvalues would
coincide, forcing $\tau_0=\tau_1$. Thus the original benchmark problem has at
least two distinct global minimizers.}
\end{proof}

\begin{remark}
Theorem \ref{thm:near-matching-unique} proves uniqueness of the matching value
near the benchmark node in the Hilbert regime.  Its proof uses the variational
uniqueness theorem and the exact spectral reduction, rather than an intrinsic
calculation of $M_T''$.  Corollary \ref{cor:strict-matching} records the sharper
route one would need for a purely spectral-matching proof: establish
differentiability and strict convexity of $M_T$ at the selected matching value.
{In this sense, strict positivity of the scalar second variation is the
spectral counterpart of the local Hilbert-space rigidity obtained from the
nodal chart.}
This remains the natural next rigidity mechanism beyond the local nodal-chart
argument.
\end{remark}

\begin{remark}
For arbitrary $T$ away from the perturbative regime of Theorem
\ref{thm:near-global-unique}, global uniqueness or multiplicity is reduced to
the minimizer structure of \eqref{eq:matching-functional}.  This is sharper than
a formal shooting assertion: it is global, variational, and stated entirely in
terms of principal-eigenvalue optimization on the two radial subintervals.
Proving strict convexity of this scalar functional, or constructing multiple
global minimizers, would be the next step toward a full global rigidity theory.
{Thus the remaining global problem is no longer a nodal-regularity issue,
but a one-dimensional spectral matching problem for the principal eigenvalues
on $(0,T)$ and $(T,R)$.}
\end{remark}

\section*{Declarations}

We declare that AI has been used for language polishing. All ideas and proofs are developed by the authors. Nevertheless, GPT was employed to check and verify our arguments, and the AI confirmed the correctness of our proofs.

\end{document}